\documentclass[reqno, 12pt]{amsart}
\pdfoutput=1
\makeatletter
\let\origsection=\section \def\section{\@ifstar{\origsection*}{\mysection}} 
\def\mysection{\@startsection{section}{1}\z@{.7\linespacing\@plus\linespacing}{.5\linespacing}{\normalfont\scshape\centering\S}}
\makeatother        

\usepackage{amsmath,amssymb,amsthm}
\usepackage{mathrsfs}
\usepackage{mathabx}\changenotsign
\usepackage{dsfont}

\usepackage{verbatim}
 
\usepackage{graphicx}
\usepackage{xcolor}
\usepackage{tikz}
\usepackage{caption}
\usepackage{subcaption}

\usetikzlibrary{arrows.meta}

\usepackage[backref]{hyperref}
\usepackage[nameinlink, capitalise, noabbrev]{cleveref}
\hypersetup{
    colorlinks,
    linkcolor={red!60!black},
    citecolor={green!60!black},
    urlcolor={blue!60!black},
    pdftitle={Infinite grids in digraphs},
    pdfauthor={Matthias Hamann, Karl Heuer}
}

\usepackage[open,openlevel=2,atend]{bookmark}

\usepackage[abbrev,msc-links,backrefs]{amsrefs} 
\usepackage{doi}

\renewcommand{\PrintDOI}[1]{\doi{#1}}

\usepackage[T1]{fontenc}
\usepackage{lmodern}
\usepackage[babel]{microtype}
\usepackage[english]{babel}

\usepackage{geometry}
\numberwithin{equation}{section}
\numberwithin{figure}{section}

\usepackage[shortlabels]{enumitem}

\let\polishlcross=\l
\def\l{\ifmmode\ell\else\polishlcross\fi}

\def\paragraph#1{
  \noindent\textbf{#1.}\enspace}

\let\emptyset=\varnothing
\let\setminus=\smallsetminus

\let\sm=\setminus

\makeatletter
\def\moverlay{\mathpalette\mov@rlay}
\def\mov@rlay#1#2{\leavevmode\vtop{   \baselineskip\z@skip \lineskiplimit-\maxdimen
   \ialign{\hfil$\m@th#1##$\hfil\cr#2\crcr}}}
\newcommand{\charfusion}[3][\mathord]{
    #1{\ifx#1\mathop\vphantom{#2}\fi
        \mathpalette\mov@rlay{#2\cr#3}
      }
    \ifx#1\mathop\expandafter\displaylimits\fi}
\makeatother

\DeclareFontFamily{U}  {MnSymbolC}{}
\DeclareSymbolFont{MnSyC}         {U}  {MnSymbolC}{m}{n}
\DeclareFontShape{U}{MnSymbolC}{m}{n}{
    <-6>  MnSymbolC5
   <6-7>  MnSymbolC6
   <7-8>  MnSymbolC7
   <8-9>  MnSymbolC8
   <9-10> MnSymbolC9
  <10-12> MnSymbolC10
  <12->   MnSymbolC12}{}
\DeclareMathSymbol{\powerset}{\mathord}{MnSyC}{180}

\let\phi=\varphi

\theoremstyle{plain}
\newtheorem{thm}{Theorem}[section]
\newtheorem{theorem}[thm]{Theorem}
\newtheorem{lemma}[thm]{Lemma}
\newtheorem{corollary}[thm]{Corollary}
\newtheorem{proposition}[thm]{Proposition}
\newtheorem{problem}[thm]{Problem}
\newtheorem{case}{Case}[]

\theoremstyle{definition}

\newtheorem{example}[thm]{Example}

\makeatletter
\newcommand\thankssymb[1]{\textsuperscript{\@fnsymbol{#1}}}
\makeatother

\begin{document}

\author[M.~Hamann]{Matthias Hamann\thankssymb{2}}
\address{Matthias Hamann, University of Hamburg, Department of Mathematics, Bundesstr. 55, 20146 Hamburg, Germany}
\email{\tt matthias.hamann@math.uni-hamburg.de}
\thanks{\thankssymb{2} Funded by the Deutsche Forschungsgemeinschaft (DFG) - Project No.\ 549406527.}

\author[K.~Heuer]{Karl Heuer}
\address{Karl Heuer, Technical University of Denmark, Department of Applied Mathematics and Computer Science, Richard Petersens Plads, Building 322, 2800 Kongens Lyngby, Denmark}
\email{\tt karheu@dtu.dk}

\title[]{Infinite grids in digraphs}

\date\today



\begin{abstract}
   Halin proved that every graph with an end $\omega$ containing infinitely many pairwise disjoint rays admits a subdivision of the infinite quarter-grid as a subgraph where all rays from that subgraph belong to~$\omega$. We will prove a corresponding statement for digraphs, that is, we will prove that every digraph that has an end with infinitely many pairwise disjoint directed rays contains a subdivision of a grid-like digraph all of whose directed rays belong to that end.
\end{abstract}

\maketitle

\section{Introduction}
\label{sec:intro}

Halin's grid theorem \cite{Halin65}*{Satz~4$'$} characterizes ends of graphs that contain infinitely many pairwise disjoint \emph{rays}, i.\,e.\ one-way infinite paths.
For this, an \emph{end} is an equivalence class of rays, where two rays are equivalent if there are infinitely many pairwise disjoint paths between them.
A \emph{subdivision} of a graph $G$ is a graph obtained from~$G$ by replacing edges by new paths between the incident vertices of that edge such that the new paths are internally disjoint and have no inner vertex in the vertex set of~$G$.
The \emph{hexagonal quarter-grid}, denoted by~$H^{\infty}$, is the graph in Figure~\ref{fig:hexgrid}.

\begin{figure}[ht]
    \centering
\begin{tikzpicture}
    \draw (0,0) -- (0,3.25);
    \foreach \y in {0,0.5,1,1.5,2,2.5,3}
        \fill (0,\y) circle (2pt);

    \draw (1,0) -- (1,3.25);
    \foreach \y in {0,0.5,1,1.5,2,2.5,3}
        \fill (1,\y) circle (2pt);
    
    \draw (2,0.5) -- (2,3.25);
    \foreach \y in {0.5,1,1.5,2,2.5,3}
        \fill (2,\y) circle (2pt);

    \draw (3,1) -- (3,3.25);
    \foreach \y in {1,1.5,2,2.5,3}
        \fill (3,\y) circle (2pt);

    \draw (4,1.5) -- (4,3.25);
    \foreach \y in {1.5,2,2.5,3}
        \fill (4,\y) circle (2pt);

    \foreach \y in {0,1,2,3}
    \draw (0,\y) -- (1,\y);

    \foreach \y in {0.5,1.5,2.5}
    \draw (1,\y) -- (2,{\y});

    \foreach \y in {1,2,3}
    \draw (2,\y) -- (3,{\y});

    \foreach \y in {1.5,2.5}
    \draw (3,\y) -- (4,{\y});

    \draw (4,2) -- (4.25,2);
    \draw (4,3) -- (4.25,3);

    \draw (2,3.5);
    \fill (2,3.5) circle (1pt);
    \draw (2,3.75);
    \fill (2,3.75) circle (1pt);
    \draw (2,4);
    \fill (2,4) circle (1pt);    

    \draw (4.25,2.25);
    \fill (4.25,2.25) circle (1pt);
    \draw (4.5,2.25);
    \fill (4.5,2.25) circle (1pt);
    \draw (4.75,2.25);
    \fill (4.75,2.25) circle (1pt);    

\end{tikzpicture}
    \caption{The hexagonal quarter-grid $H^{\infty}$.}
    \label{fig:hexgrid}
\end{figure}
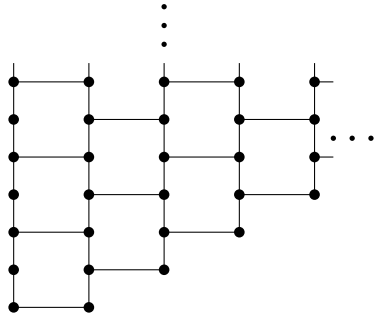

Now we are able to state Halin's theorem.

\begin{theorem}\cite{Halin65}*{Satz~4$'$}\label{thm:Halin_grid_thm}
    Whenever an undirected graph contains infinitely many pairwise disjoint and equivalent rays, then it contains a subdivision of~$H^{\infty}$.
\end{theorem}

As a first attempt to obtain an analogous result for directed graphs, \emph{digraphs} for short, Zuther \cite{Z1998}*{Theorem 3.1} proved that, if a digraph has an infinite increasing sequence of distinct ends, each of which contains a ray, then it contains a subdivision of the digraph obtained from $H^\infty$ by orienting the horizontal edges to the right and the vertical ones upwards.
Here, the \emph{ends} of digraphs are defined by Zuther~\cites{Z1997,Z1998} analogously as equivalence classes of rays and anti-rays, where a (\emph{anti-})\emph{ray} is an orientation of a one-way infinite path such that each edge is oriented towards (resp.~away from) infinity.
(We refer to Section~\ref{sec:prelims} for a precise definition and for the definition of an order on the ends).

Our main theorem is the following result, where the bidirected quarter-grid is the digraph in Figure~\ref{fig:D_1} and the reversed bidirected quarter-grid is obtained from the bidirected quarter-grid by reversing the direction of all edges.

\begin{figure}[ht]
    \centering
\begin{tikzpicture}
    \draw[-{Latex[length=2.75mm]}] (1.5,0) -- (8.25,0);
    \foreach \x in {1.5,2,2.5,3,3.5,5,5.5,7,7.5}
        \fill (\x,0) circle (2pt);

    \draw (1,-0.3) node[anchor=south] {$R_1$};

    \draw[-{Latex[length=2.75mm]}] (2,1) -- (9.25,1);
    \foreach \x in {2,2.5,3,3.5,4,4.5,5,5.5,6,6.5,7,7.5,8,8.5}
        \fill (\x,1) circle (2pt);

    \draw (1.5,0.7) node[anchor=south] {$R_2$};

    \draw[-{Latex[length=2.75mm]}] (4,2) -- (10.25,2);
    \foreach \x in {4,4.5,6,6.5,7,7.5,8,8.5,9,9.5}
        \fill (\x,2) circle (2pt);

    \draw (3.5,1.7) node[anchor=south] {$R_3$};

    \draw[-{Latex[length=2.75mm]}] (7,3) -- (11.25,3);
    \foreach \x in {7,7.5,9,9.5,10,10.5}
        \fill (\x,3) circle (2pt);

    \draw (6.5,2.7) node[anchor=south] {$R_4$};

    \draw[-{Latex[length=2.75mm]}] (10,3) -- (10,3.75);
    \draw[-{Latex[length=2.75mm]}] (10.5,3.75) -- (10.5,3);

    \draw[-{Latex[length=2.75mm]}] (2,0) -- (2,1);
    \draw[-{Latex[length=2.75mm]}] (3,0) -- (3,1);
    \draw[-{Latex[length=2.75mm]}] (5,0) -- (5,1);
    \draw[-{Latex[length=2.75mm]}] (7,0) -- (7,1);    
    
    \draw[-{Latex[length=2.75mm]}] (2.5,1) -- (2.5,0);
    \draw[-{Latex[length=2.75mm]}] (3.5,1) -- (3.5,0);
    \draw[-{Latex[length=2.75mm]}] (5.5,1) -- (5.5,0);
    \draw[-{Latex[length=2.75mm]}] (7.5,1) -- (7.5,0);    
    
    \draw[-{Latex[length=2.75mm]}] (4,1) -- (4,2);
    \draw[-{Latex[length=2.75mm]}] (6,1) -- (6,2);

    \draw[-{Latex[length=2.75mm]}] (4.5,2) -- (4.5,1);

    \draw[-{Latex[length=2.75mm]}] (6.5,2) -- (6.5,1);

    \draw[-{Latex[length=2.75mm]}] (8,1) -- (8,2);
    \draw[-{Latex[length=2.75mm]}] (8.5,2) -- (8.5,1);    

    \draw[-{Latex[length=2.75mm]}] (7,2) -- (7,3);  
    \draw[-{Latex[length=2.75mm]}] (7.5,3) -- (7.5,2);    
    
    \draw[-{Latex[length=2.75mm]}] (9,2) -- (9,3);  
    \draw[-{Latex[length=2.75mm]}] (9.5,3) -- (9.5,2);

    \draw (8.25,0.5);
    \fill (8.25,0.5) circle (1pt);
    \draw (8.5,0.5);
    \fill (8.5,0.5) circle (1pt);
    \draw (8.75,0.5);
    \fill (8.75,0.5) circle (1pt);

    \draw (9.25,1.5);
    \fill (9.25,1.5) circle (1pt);
    \draw (9.5,1.5);
    \fill (9.5,1.5) circle (1pt);
    \draw (9.75,1.5);
    \fill (9.75,1.5) circle (1pt);

    \draw (10.25,2.5);
    \fill (10.25,2.5) circle (1pt);
    \draw (10.5,2.5);
    \fill (10.5,2.5) circle (1pt);
    \draw (10.75,2.5);
    \fill (10.75,2.5) circle (1pt);

    \draw (11.25,3.5);
    \fill (11.25,3.5) circle (1pt);
    \draw (11.5,3.65);
    \fill (11.5,3.65) circle (1pt);
    \draw (11.75,3.8);
    \fill (11.75,3.8) circle (1pt);

\end{tikzpicture}
    \caption{The bidirected quarter-grid.}
    \label{fig:D_1}
\end{figure}
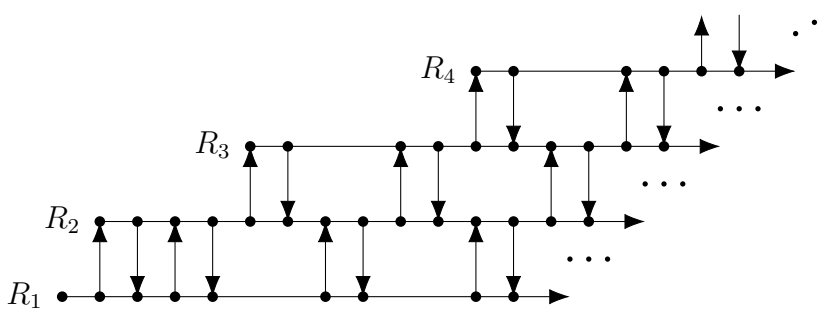

\begin{thm}\label{thm:planarGridIntro}
If $D$ is a digraph that contains an end~$\omega$ with infinitely many disjoint (anti-)rays, then there exists a subdivision of the (reversed) bidirected quarter-grid in~$D$ with all its (anti-)rays in~$\omega$.
\end{thm}

This theorem follows from a more detailed result (Theorem~\ref{thm:planarGrid}), which classifies ends containing infinitely many disjoint rays into three different types.
The classification will be in terms of recurring auxiliary digraphs which are defined on rays within such a fixed end and encode how these rays are connected to each other.
Each classifying term is characterising for the existence of a subdivision of a certain digraph; namely for the bidirected quarter-grid, for a cyclically directed quarter-grid and for a complete ray digraph (see Section~\ref{sec:subdivision} for precise definitions of these digraphs).
From this point of view, Theorem~\ref{thm:planarGrid} can be seen as an analogous result to \cite{ubiquity-2}*{Theorem~1.2} where a similar classification is proved for ends of undirected graphs.

We note that our result is stronger than Zuther's theorem, since it is possible to find bidirected quarter-grids in Zuther's grid-like digraphs.
(We refer to the end of Section~\ref{sec:subdivision} for a discussion how to find bidirected quarter-grids in Zuther's grid-like digraphs.)

Independently, Reich \cite{R2024+} recently obtained Theorem~\ref{thm:planarGridIntro} with a slightly different notion of bidirected quarter-grid, which turns out to be equivalent in that his bidirected quarter-grid contains ours as a subdivision and vice versa.

Whereas Reich's proof relies on a detailed analysis of an infinite, strongly connected auxiliary digraph allowing to prescribe a set of rays among which all vertical ones are chosen, ours looks at a sequence of finite, strongly connected digraphs and uses ideas similar to those in \cites{Heuer_full-grid, ubiquity-2}.
In order to investigate that sequence, we prove a result for finite, strongly connected digraphs (Theorem~\ref{thm:Kasper}), which might be considered interesting in its own.
It can be seen as a directed analogue of the fact that every large enough finite, connected graph contains either a vertex of high degree or a long path.

In Section~\ref{sec:thinends}, we will apply Theorem~\ref{thm:Kasper} also in order to find grid-like structures in ends that only contain finitely many pairwise disjoint rays (anti-rays), see Theorem~\ref{thm:thinEnds}.

\medskip

This paper is structured as follows. 
After introducing some terminology in Section~\ref{sec:prelims}, we prove the structural result for finite, strongly connected digraphs in Section~\ref{sec:fin_digraphs}.
In Section~\ref{sec:subdivision}, we will prove Theorem~\ref{thm:planarGridIntro} and, in Section~\ref{sec:thinends}, we will prove the result on ends with only finitely many pairwise disjoint rays or anti-rays.
We finish in Section~\ref{sec:edge grid} with a discussion of the situation when considering pairwise edge-disjoint rays or anti-rays.

\section{Preliminaries}
\label{sec:prelims}

For general facts and notation regarding graphs we refer the reader to~\cite{diestel}, regarding digraphs in particular to~\cite{bang-jensen}.

A digraph $D$ is called \emph{strongly connected} if for every two vertices $u,v \in V(D)$ there exists a path directed from~$u$ to~$v$ in~$D$.
For the sake of brevity, we call a digraph $D$ \emph{strong} if it is strongly connected, a directed cycle just a \emph{dicycle} and a directed path just a \emph{dipath}.
For a dipath $P$ containing two vertices $a$ and $b$ in this order, i.\,e.\ such that $b$ is reached from $a$ via $P$, we denote by $aPb$ the subdipath of~$P$ starting at~$a$ and ending at~$b$.
For two vertex sets $A$ and $B$, a dipath $P$ is an \emph{$A$--$B$ dipath} if $P$ starts in~$A$, ends in~$B$ and is internally disjoint from $A \cup B$.
In case $A$ or $B$ is a singleton set, we may omit the set brackets with respect to this notation.

A digraph $R$ is called a \emph{ray} if it has precisely one vertex~$v$ with out-degree~$1$ and in-degree $0$, there exists a dipath from~$v$ to each vertex of~$R$, and each vertex distinct from $v$ has out-degree $1$ and in-degree $1$.
An \emph{anti-ray} is obtained from a ray by reversing all its edges.
The vertex $v$ is called the \emph{starting vertex} (resp.\ \emph{end vertex}) of the ray (resp.\ anti-ray).
We say that a ray (or anti-ray) starts (ends) in a vertex set $A$ if it has its starting vertex (end vertex) in~$A$.
For a ray $R$ with starting vertex $v$ and some $x \in V(R)$ we denote by $Rx$ the subdipath $vRx$ of~$R$.
A \emph{tail} of~$R$ is a subray of~$R$.
If this tail starts at~$x$, then we denote it by~$xR$.
Similarly, for an anti-ray $Q$ with end vertex~$v$ and some $x\in V(Q)$, we denote by $xQ$ the subdipath $xQv$ of~$Q$ and by $Qx$ the subanti-ray of~$Q$ that ends at~$x$, which we will also call a \emph{tail} of~$Q$.

Let $Q$ and~$R$ be rays or anti-rays.
We write $Q\leq R$ if there are infinitely many pairwise disjoint $Q$--$R$ dipaths and we write $Q\sim R$ if $Q\leq R$ and $R\leq Q$.
Then $\leq$ is a preorder on the set of rays and anti-rays in a digraph~$D$ and $\sim$ is an equivalence relation on that set.
The equivalence classes of~$\sim$ are the \emph{ends} of~$D$ and we can extend the relation $\leq$ to the ends: we write $\eta\leq\omega$ for ends $\eta$ and $\omega$ if there are $Q\in\eta$ and $R\in\omega$ with $Q\leq R$.
Note that $\eta\leq\omega$ if and only if $Q\leq R$ for every $Q\in\eta$ and $R\in\omega$.
In particular, we have $\eta\leq\omega$ and $\omega\leq\eta$ if and only if $\eta=\omega$.

In \cite{HH2024+}, it was shown that an end that contains $n$ pairwise disjoint rays for all $n\in\mathbb N$ also contains infinitely many pairwise disjoint rays.
The \emph{in-degree} of an end is then the maximum number (within $\mathbb{N} \cup \{ \infty \}$) of pairwise disjoint rays in that end, and the \emph{out-degree} is the maximum number of pairwise disjoint anti-rays in that end.
We call an end \emph{thick} if it contains infinitely many pairwise disjoint rays and we call it \emph{thin} if it contains at most $n$ pairwise disjoint rays for some $n\in\mathbb N$.
If it is clear from the context, we also speak about thick and thin ends, when considering anti-rays.

\section{Unavoidable subdigraphs in strong digraphs}
\label{sec:fin_digraphs}

This section considers only finite digraphs.
The purpose of this section is to prove Theorem~\ref{thm:Kasper}, which qualitatively states that every large enough strong digraph contains an arbitrarily large strong subdigraph of a certain type.
There are only three types of these subdigraphs, which are differently structured while all maintaining strong connectivity.
Hence, Theorem~\ref{thm:Kasper} forms an analogue for digraphs to the following folklore result about undirected graphs.

\begin{proposition}\label{prop:unavoidable_1-con}
    For every $r \in \mathbb{N}$ there exists an $N \in \mathbb{N}$ such that every connected graph on at least $N$ vertices contains a path of length $r$ or a star with $r$ leaves as a subgraph.
\end{proposition}

To prepare the proof of the main result of this section, we start with the following auxiliary lemma.

\begin{lemma}\label{lem:high_deg}
For every $n \in \mathbb{N}$ there exists an $N \in \mathbb{N}$ such that for any strong digraph $D$ on at least $N$ vertices one of the following is true:
\begin{enumerate}[label = (\arabic*)]
\item There exists a dipath of length at least $n$ in $D$.
\item There exists a vertex $v \in V(D)$ with at least $n$ many out-neighbours.
\end{enumerate}
\end{lemma}

\begin{proof}
Assume there does not exist a dipath of length at least $n$ in~$D$.
Fix an arbitrary vertex $w \in V(D)$.
For $i \in \mathbb{N}$, let $D_i$ denote the set of vertices that are reached from $w$ by a shortest dipath of length $i$ in $D$.
Since $D$ is strong, we know that $V(D) = \bigcup^{n-1}_{i = 0} D_i$.
If $D$ does not contain any vertex of out-degree at least $n$, then $|D_{i+1}| \leq (n-1) \cdot |D_{i}|$ for every $i < n-1$.
Hence we have
\[|V(D)| \leq \sum^{n-1}_{i = 0} (n-1)^i = \frac{(n-1)^n - 1}{n-2} =: N .\]
Consequently, by taking $V(D) > N$ we either get a dipath of length at least $n$ or a vertex of out-degree at least $n$ in~$D$.
\end{proof}

Next we define one type of the involved strong digraphs that appear in the main result of this section.

Let $x$ and $y$ be two, potentially equal vertices.
For some non-negative integers $k, \ell$ with $k+\ell>0$, a system of $k + \ell$ internally disjoint dipaths (or dicycles in case $x = y$) where $k$ of them are $x$--$y$ dipaths and $\ell$ of them are $y$--$x$ dipaths is called a \emph{$(k, \ell)$-system of $x$--$y$ dipaths}.
We also speak about a \emph{$(k, \ell)$-system of dipaths} if we do not specify the vertices $x$ and $y$.
A $(k,\ell)$-system $\mathcal{P}$ of $x$--$y$ dipaths is called \emph{$n$-short} if $|V(P) \cup V(Q)| < n$ for all $x$--$y$ dipaths $P \in \mathcal{P}$ and all $y$--$x$ dipaths $Q \in \mathcal{P}$ in case $x \neq y$ and $k, \ell > 0$, and otherwise each dipath or dicyle in~$\mathcal{P}$ has length less than $n$.

The following theorem makes a similar statement as Theorem~\ref{thm:Kasper}.
However, the subdigraphs that are forced within each large enough strong digraph are not strong themselves in this theorem.

\begin{thm}\label{thm:Kasper_weak}
For every $n \in \mathbb{N}$ there exists an $N \in \mathbb{N}$ such that any strong digraph $D$ on at least $N$ vertices contains one of the following substructures:
\begin{enumerate}[label = (\arabic*)]
\item\label{itm:Kasper1weak} A dipath of length at least $n$.
\item\label{itm:Kasper2weak} An $n$-short $(n,0)$-system $\mathcal{P}$ of dipaths.
\end{enumerate}
\end{thm}

\begin{proof}
Suppose there does not exist a dipath of length at least $n$ in $D$.
By Lemma~\ref{lem:high_deg} we now choose $N \in \mathbb{N}$ so that the existence of a vertex $v \in V(D)$ with out-degree at least $\ell := n(n-1)^{n-3}$ is guaranteed.
Let $z \in V(D) \setminus N^+(v)$ and $P_w$ be a shortest dipath from $w$ to $z$ in $D$ for every $w \in N^+(v)$.
Furthermore we set $\mathcal{P}' = \{ P_w \mid w \in N^+(v) \}$.
Note that $z$ might also be equal to $v$ since $D$ is strong.
If there exists no $S \subseteq V(D)$ of size smaller than $n$ separating $N^+(v)$ from $z$, then there exists the desired $(n,0)$-system $\mathcal{P}$ of $v$--$z$ dipaths by Menger's Theorem.
Hence, we may assume the existence of such a set $S$.
Note that every $N^+(v)$--$S$ dipath which does not use $v$ has length at most $n-2$ as otherwise we would have a dipath of length at least $n$ from $v$ to $z$ (or to a predecessor on a dipath in case $v = z$).
By the pigeonhole principle, there exists a vertex $s \in S$ such that a set $\mathcal{P}_1 \subseteq \mathcal{P}'$ of size at least $\lceil \frac{\ell}{n-1} \rceil$ exists all whose dipaths meet $s$.
Let $D_1$ denote the digraph induced by all initial segments of dipaths of $\mathcal{P}_1$ up to the vertex $s$.
We can now repeat the argument with $s$ instead of $z$ in $D_1$.
After iterating this argument at most $n-3$ times, we have either found the desired $(n,0)$-system of dipaths via Menger's Theorem or have a vertex $s^* \in V(D)$ which is reached from $\lceil \frac{\ell}{(n-1)^{(n-3)}} \rceil \geq n$ many vertices in $N^+(v)$ via dipaths whose length is at most $1$.
This, however, gives rise to an $(n,0)$-system of $v$--$s^*$ dipaths.
\end{proof}

Before we come to the proof of the main result of this section, Theorem~\ref{thm:Kasper}, we state another type of subdigraphs that are part of the statement of the main result.

A sequence $(C_1, C_2, \ldots, C_k)$ of dicycles is called a \emph{semi-chain} if ${V(C_i) \cap V(C_j) \neq \emptyset}$ if and only if $|i-j| = 1$ for all $i, j \in \mathbb{N}$ with $1 \leq i, j \leq k$.
A semi-chain $(C_1, C_2, \ldots, C_k)$ is called \emph{$n$-narrow} for some $n \in \mathbb{N}$ if $|V(C_i)| < n$ for every $i \in \{1, \ldots, k \}$.

Now we prove the main result of this section.

\begin{thm}\label{thm:Kasper}
For every $n, k \in \mathbb{N}$ there exists an $N \in \mathbb{N}$ such that any strong digraph $D$ on at least $N$ vertices contains one of the following substructures:
\begin{enumerate}[label = (\arabic*)]
\item\label{itm:Kasper1} A dicycle on at least $n$ vertices.
\item\label{itm:Kasper2} An $n$-narrow semi-chain $(C_1, C_2, \ldots, C_k)$.
\item\label{itm:Kasper3} An $n$-short $(m,1)$-system of dipaths where $m \geq (k-1)n + 3$.
\end{enumerate}
\end{thm}

During the preparation of this article, we noticed that Theorem~\ref{thm:Kasper} is equivalent to a result in a recent preprint of Reich~\cite{R2024+arxiv}*{Corollary~1.2}.
However, our proof is much shorter and more straightforward, since the paper~\cite{R2024+arxiv} focuses on a more general version in terms of butterfly minors.

\begin{proof}[Proof of Theorem~\ref{thm:Kasper}.]
Let $n, k \in \mathbb{N}$ be fixed, and let $D$ be a strong digraph without dicycles of length at least~$n$.
By Theorem~\ref{thm:Kasper_weak} there exists an $N \in \mathbb{N}$ such that if $D$ has at least $N$ vertices, then it either contains a dipath of length $nk$ or two vertices $x, y$ together with an $nk$-short $(nk,0)$-system of $x$--$y$ dipaths. Corresponding to this, we now distinguish two cases.

\begin{case}\label{case:long_path}
$D$ contains a dipath $P$ of length $nk$.
\end{case}

Let $P = p_1p_2 \ldots p_{nk}$ and define $v_1 = p_{nk}$.
Let $Q_1$ be a dipath that is internally disjoint from $P$ from $v_1$ to a vertex $p_i =: v_2$ such that $i \in \mathbb{N}$ is as small as possible, and with respect to these properties $Q_1$ is as short as possible.
Such a dipath exists since $D$ is strong.
Since $D$ does not contain a dicycle of length at least $n$, we know that $i \geq nk - n+2$.
Now we set $C_1 = v_2Pv_1 \cup Q_1$, which is a dicycle of length less than $n$.
Since $C_1$ does not cover $V(P)$, we similarly repeat the construction of dicycles along $P$:
Let $Q_2$ be a dipath that is internally disjoint from $p_1Pv_2$ from $v_2$ to a vertex $p_j =: v_3$ such that $j \in \mathbb{N}$ is as small as possible, and with respect to these properties $Q_2$ has as few edges outside of $E(P)$ as possible.
As before, set $C_2 = v_3Pv_2 \cup Q_2$, which is another dicycle of length less than $n$.
This process is repeated until the sequence of constructed dicycles $C_1, C_2, \ldots, C_{\ell}$ covers $V(P)$.
Since each dicycle contains at most $n-1$ vertices of $P$, we know that $\ell \geq k$ holds.

We claim that $(C_1, C_2, \ldots, C_{\ell})$ is a semi-chain of dicycles.
By construction, we know that $C_i$ intersects $C_{i-1}$ and $C_{i+1}$ for every $i \in \mathbb{N}$ with $1 < i < \ell$.
Now suppose for a contradiction that $C_i$ and $C_j$ intersect for $i, j \in \mathbb{N}$ with $1 \leq i < j-1 < \ell$.
Without loss of generality, let us assume that $i$ is as big as possible with respect to this property and fixed $j$.
Note that $C_i$ and $C_j$ either intersect in $Q_i \cap Q_j$ or in $(v_{i+1}Pv_i) \cap Q_j$.
Let $c$ be the last vertex on $Q_j$ that lies in $C_i \cap C_j$.
If $c \notin V(v_{i+1}Pv_i)$, then we get a contradiction to the choice of $v_{i+1}$ since $(v_iQ_ic) \cup (cQ_jv_{j+1})$ contains a dipath from $v_i$ to a vertex on $P$ before $v_{i+1}$ that is internally disjoint from $p_1Pv_{v_i}$.
Hence, we know that $c \in V(v_{i+1}Pv_i)$.
By the choice of $i$, we know that $Q_j$ is disjoint from $C_m$ for $i < m < j-1$.
Suppose for a contradiction that $cQ_jv_{j+1}$ intersects $v_jPv_{j-1}$ before $v_{j-1}$, say in a vertex $w$.
Then we have a contradiction to the choice of $Q_j$ since the dipath $(v_jPw) \cup (wQ_jv_{j+1})$ would be a valid choice instead of $Q_j$, but has fewer edges outside of $P$ than $Q_j$.
Hence $cQ_jv_{j+1}$ does not intersect $v_jPv_{j-1}$ before $v_{j-1}$.
Now, however, we get a contradiction to the choice of $v_{i+2}$ since the dipath $(v_{i+1}Pc) \cup (cQ_jv_{j+1})$ is a valid choice for the dipath $Q_{i+1}$, but ends in $v_{j+1}$, which lies before $v_{i+2}$ on $P$.
So, $C_i$ and $C_j$ must be disjoint, which completes the proof of the claim that $(C_1, C_2, \ldots, C_{\ell})$ is a semi-chain of dicycles.

By the definition of the dicycles $C_i$, we immediately get that the semi-chain of dicycles $(C_1, C_2, \ldots, C_{\ell})$ is $n$-narrow.
This completes the proof under the assumption of Case~\ref{case:long_path}.

\begin{case}\label{case:nolong_path}
$D$ contains no dipath of length $nk$, but two vertices $x, y$ together with an $nk$-short $(nk,0)$-system $\mathcal{P}$ of $x$--$y$ dipaths.
\end{case}

In case $x=y$, we obtain the desired $n$-short $(m,1)$-system of dipaths immediately for $n \geq 3$, and for $n \leq 2$ we trivially have a dicycle of desired length since $D$ is strong.
So let us assume that $x \neq y$.
Using again that $D$ is strong, we can find a dipath $Q$ from $y$ to $x$.
By assumption, $Q$ has length less than $nk$.
Hence, $Q$ is internally disjoint to at least one dipath $P^* \in \mathcal{P}$.
But this implies that $Q$ has length at most $n-2$, as otherwise $Q \cup P^*$ would be a dicycle of length at least~$n$.
Thus, $Q$ can intersect at most $n-3$ dipaths of $\mathcal{P}$ in interior vertices.
Hence, there exist a set $\mathcal{P}' \subseteq \mathcal{P}$ containing at least $n(k-1) + 3$ many dipaths that are internally disjoint to $Q$.
As $D$ does not contain any dicycle of length at least $n$, we know that each dicycle $P' \cup Q$ for every $P' \in \mathcal{P}'$ has length less than $n$.
Therefore, $\mathcal{P}' \cup \{ Q \}$ is an $n$-short $(m,1)$-system of $x$--$y$ dipaths where $m \geq (k-1)n + 3$.
\end{proof}

\section{A grid theorem for thick ends}
\label{sec:subdivision}

Before we state the main result of this section, we need to define some digraphs.
The first digraph is built from infinitely many pairwise disjoint rays $R_1=x_1^1x_2^1\ldots$, $R_2=x_1^2x_2^2\ldots$, $\ldots$ where we add edges $x_{4j+7}^ix_{4j+1}^{i+1}$ and $x_{4j+2}^{i+1}x_{4j+8}^i$ for all $j\geq 0$ and $i\geq 1$.
Finally, we suppress all vertices $v\in V(R_i)$ with $d^-(v) = d^+(v) = 1$ that have only neighbours on~$R_i$.
We call the resulting digraph the \emph{bidirected quarter-grid}.
See Figure~\ref{fig:D_1} for a picture of the bidirected quarter-grid.
Furthermore, we call the digraph obtained from the bidirected quarter-grid after reversing all orientations of the edges the \emph{reversed bidirected quarter-grid}.

\begin{figure}[ht]
    \centering

    \begin{subfigure}[b]{0.38\linewidth}    
    \centering
\begin{tikzpicture}
    \draw[-{Latex[length=2.75mm]}] (0,0) -- (0,5.25);
    \foreach \y in {0,0.5,1,1.5,2,2.5,3,3.5,4,4.5}
        \fill (0,\y) circle (2pt);

    \draw (0,-0.75) node[anchor=south] {$R_1$};

    \draw[-{Latex[length=2.75mm]}] (1,0) -- (1,5.25);
    \foreach \y in {0,0.5,1,1.5,2,2.5,3,3.5,4,4.5}
        \fill (1,\y) circle (2pt);
   
    \draw (1,-0.75) node[anchor=south] {$R_2$};

    \draw[-{Latex[length=2.75mm]}] (2,1.5) -- (2,5.25);
    \foreach \y in {1.5,2,2.5,3,3.5,4,4.5}
        \fill (2,\y) circle (2pt);

    \draw (2,0.75) node[anchor=south] {$R_3$};

    \draw[-{Latex[length=2.75mm]}] (3,3) -- (3,5.25);
    \foreach \y in {3,3.5,4,4.5}
        \fill (3,\y) circle (2pt);
    
    \draw (3,2.25) node[anchor=south] {$R_4$};

    \foreach \y in {0,1,2,3,4}
    \draw[-{Latex[length=2.75mm]}] (0,\y) -- (1,\y);

    \foreach \y in {1.5,2.5,3.5,4.5}
    \draw[-{Latex[length=2.75mm]}] (1,\y) -- (2,{\y});

    \foreach \y in {3,4}
    \draw[-{Latex[length=2.75mm]}] (2,\y) -- (3,{\y});

    \draw[-{Latex[length=2.75mm]}] (1,0.5) -- (0,0.5);
    \draw[-{Latex[length=2.75mm]}] (2,2) -- (0,1.5);
    \draw[-{Latex[length=2.75mm]}] (3,3.5) -- (0,2.5);
    
    \draw[-{Latex[length=2.75mm]}] (1.6,5.25) -- (0,4.5);
    \draw[-{Latex[length=2.75mm]}] (3.6,4.9) -- (0,3.5);

    \draw[-{Latex[length=2.75mm]}] (3,4.5) -- (3.75,4.5);

    \draw (0.5,5.25);
    \fill (0.5,5.25) circle (1pt);
    \draw (0.5,5.5);
    \fill (0.5,5.5) circle (1pt);
    \draw (0.5,5.75);
    \fill (0.5,5.75) circle (1pt);    

    \draw (2.5,5.25);
    \fill (2.5,5.25) circle (1pt);
    \draw (2.5,5.5);
    \fill (2.5,5.5) circle (1pt);
    \draw (2.5,5.75);
    \fill (2.5,5.75) circle (1pt);    

    \draw (4,5);
    \fill (4,5) circle (1pt);
    \draw (4.25,5);
    \fill (4.25,5) circle (1pt);
    \draw (4.5,5);
    \fill (4.5,5) circle (1pt);    

\end{tikzpicture}

    \caption{An ascending cyclically directed quarter-grid.}
    \label{subfig:D_2_1}
        \end{subfigure}
       \hspace{3em}
\begin{subfigure}[b]{0.38\linewidth}
    
\centering
\begin{tikzpicture}
    \draw[-{Latex[length=2.75mm]}] (0,-0.5) -- (0,4.25);
    \foreach \y in {-0.5,0,0.5,1,1.5,2,2.5,3,3.5}
        \fill (0,\y) circle (2pt);

    \draw (0,-1.25) node[anchor=south] {$R_1$};

    \draw[-{Latex[length=2.75mm]}] (-1,0) -- (-1,4.25);
    \foreach \y in {0,0.5,1,1.5,2,2.5,3,3.5}
        \fill (-1,\y) circle (2pt);
   
    \draw (-1,-0.75) node[anchor=south] {$R_2$};

    \draw[-{Latex[length=2.75mm]}] (-2,0.5) -- (-2,4.25);
    \foreach \y in {0.5, 1, 1.5,2,2.5,3,3.5}
        \fill (-2,\y) circle (2pt);

    \draw (-2,-0.25) node[anchor=south] {$R_3$};

    \draw[-{Latex[length=2.75mm]}] (-3,1) -- (-3,4.25);
    \foreach \y in {1, 1.5, 2, 2.5, 3,3.5}
        \fill (-3,\y) circle (2pt);
    
    \draw (-3,0.25) node[anchor=south] {$R_4$};

    \foreach \y in {0}
    \draw[-{Latex[length=2.75mm]}] (0,\y) -- (-1,\y);
    \foreach \y in {0.5,1.5,2.5,3.5}
    \draw[-{Latex[length=2.75mm]}] (-1,\y) -- (0,\y);

    \foreach \y in {1,2,3}
    \draw[-{Latex[length=2.75mm]}] (-2,\y) -- (-1,{\y});

    \foreach \y in {1.5,2.5,3.5}
    \draw[-{Latex[length=2.75mm]}] (-3,\y) -- (-2,{\y});

    \draw[-{Latex[length=2.75mm]}] (0,1) -- (-2,0.5);
    \draw[-{Latex[length=2.75mm]}] (0,2) -- (-3,1);
    \draw[-{Latex[length=2.75mm]}] (0,3) -- (-3.75,1.6);
    \draw[-{Latex[length=2.75mm]}] (0,4) -- (-3.75,2.5);

    \foreach \y in {2,3}
    \draw[-{Latex[length=2.75mm]}] (-3.75,\y) -- (-3,\y);

    \draw (-0.5,4.25);
    \fill (-0.5,4.25) circle (1pt);
    \draw (-0.5,4.5);
    \fill (-0.5,4.5) circle (1pt);
    \draw (-0.5,4.75);
    \fill (-0.5,4.75) circle (1pt);    

    \draw (-2.5,4.25);
    \fill (-2.5,4.25) circle (1pt);
    \draw (-2.5,4.5);
    \fill (-2.5,4.5) circle (1pt);
    \draw (-2.5,4.75);
    \fill (-2.5,4.75) circle (1pt);    

    \draw (-4,2.75);
    \fill (-4,2.75) circle (1pt);
    \draw (-4.25,2.75);
    \fill (-4.25,2.75) circle (1pt);
    \draw (-4.5,2.75);
    \fill (-4.5,2.75) circle (1pt);    

    \draw (-4,4);
    \fill (-4,4) circle (1pt);
    \draw (-4.25,4);
    \fill (-4.25,4) circle (1pt);
    \draw (-4.5,4);
    \fill (-4.5,4) circle (1pt);    

\end{tikzpicture}
    \caption{A descending cyclically directed quarter-grid.}
 \label{subfig:D_2_2}

        \end{subfigure}
    \caption{The cyclically directed quarter-grids.}
    \label{fig:D_2}
\end{figure}
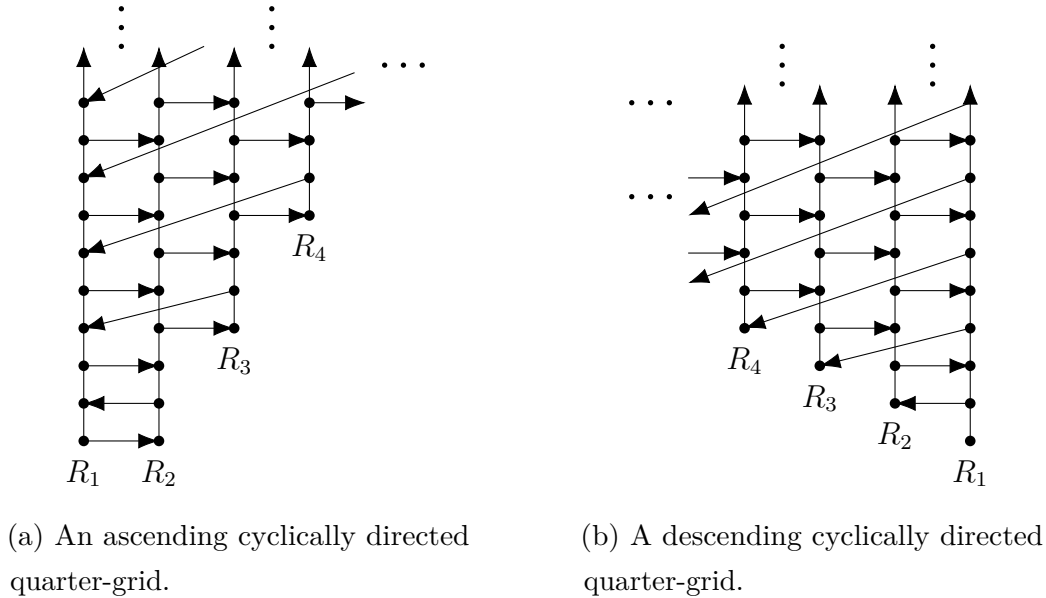

In order to define two further digraphs, we start again with infinitely many pairwise disjoint rays $R_1=x_1^1x_2^1\ldots$, $R_2=x_1^2x_2^2\ldots$, $\ldots$.
For the first one, we add the following edges: $x_j^1x_j^2$ for all odd~$j$, $x_{j+3}^ix_j^{i+1}$ for all $i\geq 2$ and all odd $j$, and $x_2^ix_{2(i-1)}^1$ for all $i\geq 2$.
This digraph is called the \emph{ascending cyclically directed quarter-grid}.
For the second one, we add the following edges: $x_{j}^{i+1}x_{j+1}^i$ for all even~$j$ and all $i\geq 1$, and $x_{2i}^1x_1^{i+1}$ for all $i\geq 1$.
We call this digraph the \emph{descending cyclically directed quarter-grid}.
We may refer to both of these digraphs simply as \emph{cyclically directed quarter-grids}.
See Figure~\ref{fig:D_2} for pictures of both cyclically directed quarter-grids.

A \emph{complete ray digraph} on a set of infinitely many pairwise disjoint rays rays $R_1=x_1^1x_2^1\ldots$, $R_2=x_1^2x_2^2\ldots$, $\ldots$ is a digraph that is obtained by the rays $R_i$ together with infinitely many disjoint dipaths in both directions between every two of them such that these dipaths do not meet any other ray~$R_j$, such that all these additional dipaths are pairwise disjoint and such that the starting vertex of $R_i$ is the end vertex of an $R_1$--$R_i$ dipath.

In the following lemma we prove that the previously defined digraphs are all one-ended.

\begin{lemma}\label{lem:one-ended}
    The bidirected quarter-grid, the cyclically directed quarter-grids and the complete ray digraphs have each precisely one end.
\end{lemma}

\begin{proof}
    We shall only prove the statement for the bidirected quarter-grid since the proofs for the other digraphs are very similar.
    Obviously, $R \leq R_1$ holds for every ray $R$ in the bidirected quarter-grid.
    For the converse, note that $R_1 \leq R_i$ holds for every $i \geq 1$.
    Hence, if a ray $R$ contains infinitely many vertices from some $R_i$, we immediately get that $R_1 \leq R$.
    Therefore, we may assume that $R$ contains vertices from $R_i$ for infinitely many $i \in \mathbb{N}$.
    Now note that for each $i \in \mathbb{N}$ there exist an $R_1$--$x^i_1$ dipath such that all these dipaths are pairwise disjoint.
    From this it can easily deduced that $R_1 \leq R$ holds.    
\end{proof}

We continue with the definition of another auxiliary digraph.
Let $D$ be a digraph and $\mathcal R$ be a set of disjoint rays in~$D$.
The \emph{auxiliary ray digraph} $D_{\mathcal R}$ has $\mathcal R$ as its vertex set and an edge from $R_1$ to~$R_2$ if there are infinitely many disjoint $R_1$--$R_2$ dipaths that do not meet any ray of $\mathcal R\sm\{R_1,R_2\}$.
If $\mathcal R\subseteq\omega$ for an end $\omega$ of~$D$, then we also call it an \emph{auxiliary $\omega$-ray digraph}.
Note that if $\mathcal R$ consists of finitely many pairwise equivalent rays, then $D_{\mathcal R}$ is strong.

Now we state the main result of this section, from which we later deduce Theorem~\ref{thm:planarGridIntro}.

\begin{thm}\label{thm:planarGrid}
Let $D$ be a digraph with an end $\omega$ that contains infinitely many disjoint rays $R_1,R_2,\dots$.
Then the following holds.
\begin{enumerate}[label = \rm (\roman*)]
\item\label{itm:planarGrid1} The digraph $D$ contains a subdivision of the bidirected quarter-grid all of whose rays are in~$\omega$, if and only if the auxiliary $\omega$-ray digraphs on $R_1,\ldots,R_i$ for all $i\in\mathbb N$ contain semi-chains of unbounded size.
\item\label{itm:planarGrid2} The digraph $D$ contains a subdivision of the cyclically directed quarter-grid all of whose rays are in~$\omega$, if and only if the auxiliary $\omega$-ray digraphs on $R_1,\ldots,R_i$ for all $i\in\mathbb N$ contain cycles of unbounded length.
\item\label{itm:planarGrid3} The digraph $D$ contains a subdivision of the complete ray digraph all of whose rays are in~$\omega$, if and only if the auxiliary $\omega$-ray digraphs on $R_1,\ldots,R_i$ for all $i\in\mathbb N$ contain $(m,1)$-systems of dipaths for all~$m$.
\end{enumerate}

Additionally, the finite auxiliary $\omega$-ray digraphs on $R_1,\ldots,R_i$ for all $i\in\mathbb N$ meet the conditions of one of \ref{itm:planarGrid1}--\ref{itm:planarGrid3}. 
\end{thm}

\begin{proof}
Clearly, the definitions of the bidirected quarter-grid, cyclically directed quarter-grid and the complete ray digraph imply the existence of suitable sequences of auxiliary digraphs.

\medskip

Let $R_1,R_2,\ldots$ be infinitely many pairwise disjoint rays in~$\omega$.
We define a sequence $(D_i)_{i\in\mathbb N}$ of auxiliary $\omega$-ray digraphs where $D_i$ has vertex set $\{ R_j\mid 1\leq j\leq i\}$.
Let $(i_j)_{j\in\mathbb N}$, $(k_j)_{j\in\mathbb N}$ and $(n_j)_{j\in\mathbb N}$ be increasing sequences going to~$\infty$ such that $N(k_p,n_p)\leq i_p$ for the constant $N(k_p,n_p)$ from Theorem~\ref{thm:Kasper}.
Thus, we can apply Theorem~\ref{thm:Kasper} to each~$D_{i_p}$ for $k_p$ and $n_p$ and obtain that $D_{i_p}$ either contains a dicycle of length~$n_p$, an $n_p$-narrow semi-chain on~$k_p$ dicycles or an $n_p$-short $((k_p-1)n_p+3,1)$-system of dipaths.
This implies that we can find a subfamily $(D_i)_{i\in I}$ for a strictly increasing sequence $I$ in~$\mathbb N$ such that either all $D_i$ contain dicycles whose lengths are strictly increasing or all $D_i$ contain semi-chains whose numbers of dicycles are strictly increasing or all $D_i$ contain $(\ell,1)$-systems of dipaths for strictly increasing~$\ell$.
This completes the proof of the additional statement of the theorem.

\medskip

Let us now prove the remaining implication of~\ref{itm:planarGrid2}.
For this, we assume that all elements of $(D_i)_{i\in I}$ contain dicycles and that the lengths of those are strictly increasing.
We may assume that $I=(n_i)_{i\in\mathbb N}$ is such that $n_{i+1}\geq n_i^2$ and such that $D_{n_{i+1}}$ contains a dicycle of length at least~$n_i$.
Recall that $D_{n_{i+1}}$ contains precisely $n_{i+1}$ many vertices.
We will define, for every element $n_j$ of~$I$, a dicycle $C_{n_j}$ in $D_{n_{j+1}}$ of length at least $n_j$, and a sequence $(P_i^j)_{1\leq i\leq n_{j-1}}$ of dipaths in $D_{n_{j+1}}$ that has the following properties.
\begin{enumerate}[label = (\arabic*)]
\item\label{itm:planarGridCycle1} Each $P_i^j$ starts at a vertex of $C_{n_{j-1}}$, seen as vertices of $D_{n_{j+1}}$, and ends at a vertex of~$C_{n_j}$.
\item\label{itm:planarGridCycle2} $P_i^j$ does not contain any end vertex of $P_k^j$ for $k<i$.
\item\label{itm:planarGridCycle3} $P_i^j$ does not contain any starting vertex of $P_k^j$ for $k>i$.
\end{enumerate}
Let us assume that we have constructed for a finite sequence $(n_0,\ldots, n_j)$ dicycles $C_{n_k}$, for $k\leq j$, and sequences of dipaths $P_i^k$, for $k\leq j$ and $1\leq i\leq n_{j-1}$.
Let $C_{n_{j+1}}$ be a dicycle of length at least $n_{j+1}$ in $D_{n_{j+2}}$, which exists by our assumption.
By the lengths of $C_{n_j}$ and $C_{n_{j+1}}$ and as $D_{n_{j+2}}$ is strongly connected, there exists a sequence $(P_1^{j+1},\ldots, P_{n_j}^{j+1})$ of dipaths starting in a vertex set in $V(C_{n_j})$ that contains the end vertices of the dipaths $P_i^{j}$ and that end on $C_{n_{j+1}}$ such that \ref{itm:planarGridCycle1}--\ref{itm:planarGridCycle3} hold for this sequence.

Let us now construct an infinite sequence $(I_j)_{j\in\mathbb N}$ of subsequences of~$I$ such that $I_j$ and $I_{j+1}$ have their first $j$ elements in common.
For this, we start with $I_1:=I$ and assume that we have already constructed the sequences $I_1,\ldots,I_{j-1}$.
Let $v_1,\ldots, v_{n_{j-1}}$ be the starting vertices of the dipaths $P_i^j$ on $C_{n_{j-1}}$ in the cyclic order given by the cycle.
For every $k\geq j$, let $x_1^{j,k},\ldots, x_{n_{j-1}}^{j,k}$ be the vertices on $C_{n_k}$ that are end vertices of dipaths $P_i^{k}$ such that $x_i^{j,k}$ is obtained by starting at~$v_i$ and following the $P_{i'}^j$ starting at~$v_i$ in $D_{n_{j+1}}$ and at its end vertex on $C_{n_{j}}$ follow $P_{i''}^{j+1}$ in $D_{n_{j+2}}$ to its end vertex on $C_{n_{j+1}}$ and so on until we reach $C_{n_k}$ in $D_{n_{k+1}}$ via the dipath $P_{i^\circ}^{k}$ at its end vertex $x_i^{j,k}$.
This is possible by~\ref{itm:planarGridCycle1}.
Let $\sigma_{j,k}$ be a permutation of $\{1,\ldots,n_{j-1}\}$ such that $x_{\sigma_{j,k}(1)}^{j,k},\ldots, x_{\sigma_{j,k}(n_{j-1})}^{j,k}$ is the cyclic order of these vertices on $C_{n_k}$.
Then there exists a permutation $\sigma_j$ of $\{1,\ldots,n_{j-1}\}$ such that infinitely many elements $k$ of $I_{j-1}$ satisfy $\sigma_j=\sigma_{j,k}$.
We modify $I_{j-1}$ to obtain $I_{j}$ such that after the $(j-1)$-st entry we take a subsequence consisting only of such elements~$k$.

Let $I_\infty=(m_i)_{i\in\mathbb N}$ be the subsequence of~$I$ whose first $j$ entries are the first $j$ entries of~$I_j$ for every $j\in\mathbb N$.
Let us now construct a subdigraph of~$D$ based on the sequences that we have constructed.
While this will, generally, not be a subdivision of one of the cyclically directed quarter-grids, the constructed subgraph will contain a subdivision of one of those two digraphs.

Let us fix some $R^* \in \omega$.
We start by taking large enough finite initial segments of the rays that are the vertices $x^{1,m_1}_1,\ldots, x^{1,m_1}_{n_0}$ in~$D_{m_2}$ such that there exists a set $\mathcal{P}^*$ of dipaths, one from each of those initial segments to~$R^*$ and vice versa.
We shall later during our construction find such dipaths again, so we denote this step of our construction briefly by \emph{connecting to}~$R^*$.
Note that we do not include the dipaths from $\mathcal{P}^*$ in our construction of the desired subdivisions.
We only need to ensure their existence in~$D$.
Let $X\subseteq V(D)$ be a finite vertex set.
Now we continue by prolonging our already chosen initial segments of the rays $x^{1,m_1}_1,\ldots, x^{1,m_1}_{n_0}$ such that we can find a set of disjoint dipaths all avoiding~$X$, one from each initial segment to the initial segment of its successor in the cyclic order induced by $C_{m_1}$ with the property that on each ray, except for $x^{1,m_1}_1$, we first have the end vertex of the dipath from the cyclic predecessor before we have the starting vertex of the dipath to the cyclic successor.
We may assume that for $x^{1,m_1}_1$ it is the other way, that is, this ray first contains the starting vertex of the dipath to its cyclic successor before it contains the end vertex of the other dipath coming from its cyclic predecessor.
We will use such a construction later once more, so will simply say that we have \emph{cyclically connected} the segments \emph{avoiding} $X$, when referring to this particular construction.

Let $X \subseteq V(D)$ be a finite vertex set.
Then we say that we \emph{follow} a dipath $P_i^j$ in~$D$ and \textit{avoid} $X$ if we run on the ray $S_1$ corresponding to the first vertex of $P_i^j$ until we can follow a dipath $P_{1,2}$ to a tail $T_2$ of the ray corresponding to the second vertex of $P_i^j$ such that $P_{1,2}$ and $T_2$ are disjoint from $X$, and $P_{1,2}$ intersects only $S_1$ and $T_2$, but no other ray from $D_{n_{j+1}}$; furthermore, we continue this along $P_i^j$ while guaranteeing that all dipaths $P_{\ell, \ell+1}$ are chosen to be pairwise disjoint, until we reach the ray corresponding to the last vertex of~$P_i^j$.
Let $k,\ell\in\mathbb N$ such that $m_1=n_k$ and $m_2=n_\ell$.
We say that we \emph{reroute} from $C_{m_1}$ to $C_{m_2}$ if we follow each dipath $P_i^{n_{k+1}}$ to $C_{n_{k+1}}$ and avoid the finite construction we made so far, i.e.\ segments of rays together with dipaths used for cyclically connecting them, then follow the dipaths $P_i^{n_{k+2}}$ avoiding the vertices of our finite construction and so on until we follow the dipaths $P_i^{n_{\ell}}$ avoiding the vertices of our finite construction.

We continue our construction by taking the finite digraph obtained from connecting segments of the rays $x^{1,m_1}_1,\ldots, x^{1,m_1}_{n_0}$ to~$R^*$, cyclically connecting the obtained segments twice each time avoiding everything built so far and rerouting to~$C_{m_2}$.
For every ray $Q$ on~$C_{m_2}$ that does not contain an end vertex of some rerouted dipath, we add a new vertex $q \in V(Q)$ to our construction such that $qQ$ avoids everything that we constructed so far.

We recursively continue this construction by repeating the following small steps: first, we connect the dipaths to~$R^*$ via dipaths that are disjoint from all dipaths of previous steps of connecting to~$R^*$, then we cyclically connect the dipaths twice with respect to the cycle $C_{m_j}$ each time avoiding everything built so far, then we reroute from $C_{m_j}$ to $C_{m_{j+1}}$ and, lastly, we add dipaths of the new rays from $C_{m_{j+1}}$.
We call this combination of these steps a \emph{big} step.
After having performed big steps along the whole sequence $I_\infty$, we denote the resulting subdigraph of~$D$ by~$D'$.

If we remove from $D'$ the dipaths added in the step of cyclically connecting segments of the rays, then we are left with infinitely many pairwise disjoint rays.
All these rays are contained in~$\omega$ since we always connected disjoint segments of those rays to~$R^*$ resulting in systems of infinitely many disjoint dipaths to and from~$R^*$, which forces them to be equivalent to~$R^*$.
By the choice of~$I_\infty$, in each big step of our construction, these rays keep their cyclic order.
We split up our cyclic order of the rays by saying that the ray that was started in the very first step with a subdipath of $x^{1,m_1}_1$ is our smallest ray and all other rays are ordered above it corresponding to the cyclic order.
That way, we obtain a linear order.
Since there are infinitely many pairwise disjoint rays, we either contain an infinite strictly increasing set of rays or an infinite strictly decreasing set of rays.
In the first case, since we cyclically connected the segments of the resulting rays in each big step twice, we can ensure that every increasing subsequence is cyclically connected once in that big step.
Hence, we immediately obtain a subdivision of the ascending cyclically directed quarter-grid by restricting to that sequence and keeping the cyclically connecting dipaths across potentially skipped rays.
In the second case, we similarly obtain a subdivision of the descending cyclically directed quarter grid. 
Finally, it follows from Lemma~\ref{lem:one-ended} that the constructed subdivisions of the cyclically directed quarter-grids have all rays in~$\omega$.

\medskip

Let us now prove the missing implication of~\ref{itm:planarGrid1}.
We assume that all elements of $(D_i)_{i\in I}$ contain semi-chains of increasing numbers of dicycles, that is, there exist strictly increasing sequences $I = (n_j)_{j\in\mathbb N}$ and $(k_j)_{j\in\mathbb N}$ such that $D_{n_{j}}$ contains a semi-chain of~$k_j$ dicycles.
Each of these semi-chains is $n$-narrow for a certain $n\in \mathbb N$, but we shall not make use of this, so we can drop that information.
We may assume that $k_{j+1}\geq k_j^2$.
Similarly to the previous case, we will define for every $k_j$ a semi-chain $\mathcal S_C^j(C_1^j,\ldots, C_{k_j}^j)$ of $k_j$ dicycles $C_1^j,\ldots, C_{k_j}^j$ and a sequence $(P_i^j)_{1\leq i\leq k_{j-1}}$ of dipaths in $D_{n_{j}}$ that has the following properties.
\begin{enumerate}[label = (\arabic*)]
\setcounter{enumi}{3}
\item\label{itm:planarGridSC1} Each $P_i^j$ starts at a vertex of a different $C^{j-1}_i$, seen as vertices of $D_{n_{j}}$, and ends at a vertex of a different~$C^j_\ell$ for odd~$\ell$.
\item\label{itm:planarGridSC2} $P_i^j$ does not contain any end vertex of $P_k^j$ for $k<i$.
\item\label{itm:planarGridSC3} $P_i^j$ does not contain any starting vertex of $P_k^j$ for $k>i$.
\end{enumerate}
Let us assume that, for a finite sequence $(k_0,\ldots,k_j)$, we have constructed the semi-chains $\mathcal S_C^\ell(C_1^\ell,\ldots, C_{k_\ell}^\ell)$ of dicycles and dipaths $P_i^\ell$ for $1\leq\ell\leq j$ and  $1\leq i\leq k_{\ell - 1}$ as claimed above.
Let $\mathcal S_C^{j+1}(C_1^{j+1},\ldots, C_{k_{j+1}}^{j+1})$ be a semi-chain of $k_{j+1}$ dicycles in~$D_{n_{j+1}}$.
By the choices of $k_j$ and $k_{j+1}$ and as $D_{n_{j+1}}$ is strongly connected, there exists a sequence $(P_1^{j+1},\ldots, P_{k_j}^{j+1})$ of dipaths starting at vertices of different dicycles $C_i^j$ that contain the end vertices of the dipaths $P_\ell^{j-1}$ and that end on different dicycles $C_i^{j+1}$ for odd $i$ such that \ref{itm:planarGridSC1}--\ref{itm:planarGridSC3} hold for this sequence.
We may assume that the starting vertices always avoid the dicycle $C_{i-1}^j$ and the analogue is true for the end vertices.

Let us now define a sequence $(I_j)_{j\in\mathbb N}$ of subsequences of~$I$ such that $I_j$ and $I_{j+1}$ have their first $j$ elements in common.
We follow the definition of the sequence as in the previous case except that we consider a total order induced by the semi-chains of dicycles instead of a cyclic order.
Again, we let $I_\infty$ be the subsequence of~$I$ whose first $j$ elements coincide with the first $j$ elements of~$I_j$.
We will now construct a subdigraph of~$D$ of which we will show later that it contains a subdivision of the bidirected quarter-grid.

As in the previous case, we let $x_i^{j,m_j}$ denote the end vertices after starting at the first vertex of $P_i^j$ and following at its end vertex the dipath $P_{i'}^{j+1}$ in $D_{j+1}$ and so on until we reach the end vertex $x_i^{j,m_j}$ of a dipath $P_{i''}^{m_j}$.
Now, precisely as in the previous case, we take suitable starting dipaths of the rays corresponding to the $x_i^{1,m_1}$ and \emph{connect} them to some fixed $R^* \in \omega$.
We continue by extending those starting dipaths of the rays $Q_i$ corresponding to the $x_i^{1,m_1}$ and join every two consecutive ones in the order that we took in the step of defining $I_1$ by disjoint dipaths in both directions that avoid all other rays belonging to some $x_\ell^{1,m_1}$.
We do this by starting at~$Q_1$ and first joining it via a dipath to~$Q_2$ and then finding a disjoint dipath from~$Q_2$ to~$Q_1$ such that the end vertex of the first dipath lies before the starting vertex of the second dipath on~$Q_2$ and the first vertex of the first dipaths lies before the end vertex of the second dipath on~$Q_1$.
We continue this until we reach the maximal element such that for each $i$ the vertices on~$Q_i$ that lie on dipaths between $Q_{i-1}$ and $Q_i$ lie before the vertices on dipaths between $Q_i$ and~$Q_{i+1}$.
We call this \emph{linearly connecting} segments of the rays and say that it \emph{avoids} a finite set $X\subseteq V(D)$ if none of the dipaths intersects~$X$.

Exactly the same way as in the previous case, we are \emph{rerouting} from semi-chains of dicycles to larger ones.
After that, for every odd $i$ such that $C_{i}^{m_j}$ does not contain a ray that contains an end vertex of some rerouted dipath, we choose a ray $Q$ on~$C_{i}^{m_j}$ and add a new vertex $q \in V(Q)$ to our construction such that $qQ$ avoids everything that we constructed so far.

We recursively repeat these four steps of connecting to~$R^*$ via dipaths that are disjoint from all dipaths of previous steps of connecting to~$R^*$, linearly connecting segments of the rays avoiding everything built so far, rerouting them and then adding remaining ones.
Iterating this along all of $I_{\infty}$, we obtain the subdigraph $D'$ of~$D$.

Removing the dipaths added in the step of linearly connecting segments of rays leads to infinitely many pairwise disjoint rays in~$\omega$ that are arranged in a linear order.
Thus, there exists an infinite strictly ascending or strictly descending sequence.
While the first case directly leads to a subdivision of the bidirected quarter-grid by restricting to that sequence and keeping the linearly connecting dipaths across potentially skipped rays, we have to remove some of the linearly connecting dipaths in the second case.
Lemma~\ref{lem:one-ended} ensures again that the constructed subdivision of the bidirected quarter-grid has all rays in~$\omega$.

\medskip

Let us now finish the proof of~\ref{itm:planarGrid3}.
We assume that $(D_i)_{i\in \mathbb N}$ contains $(\ell,1)$-systems of $x_i$--$y_i$ dipaths for vertices $x_i,y_i\in V(D_i)$ for increasing~$\ell$.
Similarly as in the case before, all these $(\ell,1)$-systems of $x_i$--$y_i$ dipaths are $n$-short for certain $n \in \mathbb N$, but we shall not need that additional information.
Let $I = (n_i)_{i \in \mathbb N}$ and $(\ell_i)_{i \in \mathbb N}$ be strictly increasing sequences of natural numbers where we may assume that the digraph $D_{n_i}$ contains an $(\ell_i,1)$-system of $x_i$--$y_i$ dipaths and that $\ell_{i+1}\geq 2\ell_i$.
Let $H_{n_i}$ be the subdigraph of~$D_{n_i}$ formed by the $x_{n_i}$--$y_{n_i}$ dipaths and the $y_{n_i}$--$x_{n_i}$ dipath of the $(\ell_i,1)$-system of $x_{n_i}$--$y_{n_i}$ dipaths.
For every $j\in\mathbb N$, we define a sequence $(P_1^j,\ldots, P_{\ell_j}^j)$ of dipaths in~$D_{n_j}$ with the following properties.
\begin{enumerate}[label = (\arabic*)]
\setcounter{enumi}{6}
\item\label{itm:planarGridSys1} Each $P_i^j$ starts at an inner vertex of a different $x_{j-1}$--$y_{j-1}$ dipath in $H_{n_{j-1}}$, seen as vertices of $D_{n_j}$, and ends at an inner vertex of a different $x_{j}$--$y_{j}$ dipath.
\item\label{itm:planarGridSys2} $P_i^j$ does not contain any end vertex of $P_k^j$ for $k<i$.
\item\label{itm:planarGridSys3} $P_i^j$ does not contain any starting vertex of $P_k^j$ for $k>i$.
\end{enumerate}
These dipaths exist by the choices of~$\ell_j$ and as $D_{n_{j}}$ is strongly connected.

Contrary to the previous two cases, we do not need to refine our sequence in this case.
Instead, we can directly start with the construction of a complete ray digraph.
For that, we first connect starting dipaths of the rays belonging to the starting vertices of the dipaths $P_i^1$ to~$R^*$ as in the previous cases.
Then we \emph{completely connect} those segments of rays: between every two we add a dipath in each direction such that these dipaths do not intersect and such that each of these dipaths does not intersect with the rays from the starting vertices of the other dipaths.
We obtain these desired dipaths similarly as in the previous cases by first finding suitable dipaths in $H_{n_1}$, and then translating them to dipaths in~$D$.

Then we can \emph{reroute} the starting dipaths along the dipaths $P_i^1$ in $D_{n_2}$ to the rays belonging to the end vertices of those dipaths.
As the fourth step, let $T$ be a $x_1$--$y_1$ dipath in $H_{n_1}$ with inner vertices that does not contain a ray which contains an end vertex of some rerouted dipath.
For every such $T$ we choose a ray $Q$ on~$T$ that corresponds to an inner vertex and add a new vertex $q \in V(Q)$ to our construction such that $qQ$ avoids everything that we constructed so far.
Again, we recursively repeat these four steps for every $j\in\mathbb N$, where the connecting to~$R^*$ happens via dipaths that are disjoint from all dipaths of previous steps of connecting to~$R^*$.
The resulting digraph contains a complete ray digraph which has all rays in~$\omega$ by construction and Lemma~\ref{lem:one-ended}.
\end{proof}

Now we deduce Theorem~\ref{thm:planarGridIntro} from the previous theorem.

\begin{proof}[Proof of Theorem~\ref{thm:planarGridIntro}]
We shall prove the theorem only for the case that there exist infinitely many pairwise disjoint rays in~$\omega$.
The other case concerning anti-rays follows from that proof applied to the digraph obtained from $D$ where all edges are reversed.

Since $\omega$ contains infinitely many pairwise disjoint rays, there exists an infinite sequence of finite auxiliary $\omega$-ray digraphs of strictly increasing size.
Therefore, Theorem~\ref{thm:planarGrid} implies that $D$ contains one of the following digraphs as subdivision all of whose rays are in~$\omega$: either a bidirected quarter-grid, or one of the cyclically directed quarter grids or a complete ray digraph.

\begin{figure}[ht]
    \centering
       \begin{subfigure}[b]{0.38\linewidth}
    \centering
\begin{tikzpicture}

    \draw[-{Latex[length=2.75mm]}] (0,0) -- (0,6.75);
    \draw[-{Latex[length=2.75mm]}] (1,0) -- (1,6.75);
    \draw[-] (2,1.5) -- (2,6.5);
    \draw[-] (3,3) -- (3,6.5);
    \draw[-{Latex[length=2.75mm]}] (4,4.5) -- (4,6.75);

    \foreach \y in {0,3,4,5,6}
    \draw[-{Latex[length=2.75mm]}] (0,\y) -- (1,\y);

    \foreach \y in {4,6}
    \draw[-{Latex[length=2.75mm]}] (2,\y) -- (3,{\y});

    \draw[-{Latex[length=2.75mm]}] (2,2) -- (0,1.5);
    \draw[-{Latex[length=2.75mm]}] (3,3.5) -- (0,2.5);
    \draw[-{Latex[length=2.75mm]}] (4,5) -- (0,3.5);
    \draw[-{Latex[length=2.75mm]}] (4.25,6.25) -- (0,4.5);
    \draw[-{Latex[length=2.75mm]}] (2.45,6.5) -- (0,5.5);

    \draw[blue, very thick, -{Latex[length=2.75mm]}] (1,0) -- (1,1.5);
    \draw[blue, very thick, -{Latex[length=2.75mm]}] (1,1.5) -- (2,1.5);
    \draw[blue, very thick, -{Latex[length=2.75mm]}] (2,1.5) -- (2,3);
    \draw[blue, very thick, -{Latex[length=2.75mm]}] (2,3) -- (3,3);
    \draw[blue, very thick, -{Latex[length=2.75mm]}] (3,3) -- (3,4.5);
    \draw[blue, very thick, -{Latex[length=2.75mm]}] (3,4.5) -- (4,4.5);
    \draw[blue, very thick, -{Latex[length=2.75mm]}] (4,4.5) -- (4,6);
    \draw[blue, very thick, -{Latex[length=2.75mm]}] (4,6) -- (4.75,6);
    
    \draw[red, very thick, -{Latex[length=2.75mm]}] (0,0.5) -- (0,2);
    \draw[red, very thick, -{Latex[length=2.75mm]}] (0,2) -- (1,2);
    \draw[red, very thick, -{Latex[length=2.75mm]}] (1,2) -- (1,3.5);
    \draw[red, very thick, -{Latex[length=2.75mm]}] (1,3.5) -- (2,3.5);
    \draw[red, very thick, -{Latex[length=2.75mm]}] (2,3.5) -- (2,5);
    \draw[red, very thick, -{Latex[length=2.75mm]}] (2,5) -- (3,5);
    \draw[red, very thick, -{Latex[length=2.75mm]}] (3,5) -- (3,6.75);
    
    \draw[orange, very thick, -{Latex[length=2.75mm]}] (1,4) -- (1,5.5);
    \draw[orange, very thick, -{Latex[length=2.75mm]}] (1,5.5) -- (2,5.5);
    \draw[orange, very thick, -{Latex[length=2.75mm]}] (2,5.5) -- (2,6.75);
    
    \draw[cyan, very thick, -{Latex[length=2.75mm]}] (1,0.5) -- (0,0.5);
    \draw[cyan, very thick, -{Latex[length=2.75mm]}] (0,1) -- (1,1);

    \draw[cyan, very thick, -{Latex[length=2.75mm]}] (1,1.5) -- (1,2);
    \draw[cyan, very thick, -{Latex[length=2.75mm]}] (1,2.5) -- (2,2.5);
    
    \draw[cyan, very thick, -{Latex[length=2.75mm]}] (3,4.5) -- (3,5);
    \draw[cyan, very thick, -{Latex[length=2.75mm]}] (3,5.5) -- (4,5.5);

    \draw[cyan, very thick, -{Latex[length=2.75mm]}] (1,3.5) -- (1,4);
    \draw[cyan, very thick, -{Latex[length=2.75mm]}] (1,4.5) -- (2,4.5);

    \foreach \y in {0,0.5,1,1.5,2,2.5,3,3.5,4,4.5,5,5.5,6}
        \fill (0,\y) circle (2pt);

    \foreach \y in {0,0.5,1,1.5,2,2.5,3,3.5,4,4.5,5,5.5,6}
        \fill (1,\y) circle (2pt);

    \foreach \y in {1.5,2,2.5,3,3.5,4,4.5,5,5.5,6}
        \fill (2,\y) circle (2pt);

    \foreach \y in {3,3.5,4,4.5,5,5.5,6}
        \fill (3,\y) circle (2pt);

    \foreach \y in {4.5,5,5.5,6}
        \fill (4,\y) circle (2pt);

    \draw (0.5,6.75);
    \fill (0.5,6.75) circle (1pt);
    \draw (0.5,7);
    \fill (0.5,7) circle (1pt);
    \draw (0.5,7.25);
    \fill (0.5,7.25) circle (1pt);    

    \draw (2.5,6.75);
    \fill (2.5,6.75) circle (1pt);
    \draw (2.5,7);
    \fill (2.5,7) circle (1pt);
    \draw (2.5,7.25);
    \fill (2.5,7.25) circle (1pt);    

    \draw (4.5,6.5);
    \fill (4.5,6.5) circle (1pt);
    \draw (4.75,6.5);
    \fill (4.75,6.5) circle (1pt);
    \draw (5,6.5);
    \fill (5,6.5) circle (1pt);

\end{tikzpicture}
        \end{subfigure}
       \hspace{3em}
        \begin{subfigure}[b]{0.38\linewidth}
    \centering

\begin{tikzpicture}
    \draw[-] (0,-0.5) -- (0,4);

    \draw[-] (-1,0) -- (-1,4);

    \draw[-] (-3,1) -- (-3,4);

    \draw[-{Latex[length=2.75mm]}] (-4,1.5) -- (-4,4.25);

    \foreach \y in {2.5}
    \draw[-{Latex[length=2.75mm]}] (-1,\y) -- (0,\y);

    \foreach \y in {3}
    \draw[-{Latex[length=2.75mm]}] (-2,\y) -- (-1,{\y});

    \foreach \y in {1.5,3.5}
    \draw[-{Latex[length=2.75mm]}] (-3,\y) -- (-2,{\y});

    \draw[blue, very thick, -{Latex[length=2.75mm]}] (0,1) -- (-2,0.5);
    \draw[red, very thick, -{Latex[length=2.75mm]}] (0,2) -- (-3,1);
    \draw[orange, very thick, -{Latex[length=2.75mm]}] (0,3) -- (-4,1.5);
    \draw[-{Latex[length=2.75mm]}] (0,4) -- (-4.75,2.1);

    \draw[-] (-2,2) -- (-2,2.5);

    \foreach \y in {2.5,3.5}
    \draw[-{Latex[length=2.75mm]}] (-4.75,\y) -- (-4,\y);

    \draw[blue, very thick, -{Latex[length=2.75mm]}] (0,-0.5) -- (0,1);
    \draw[blue, very thick, -{Latex[length=2.75mm]}] (-2,0.5) -- (-2,2);
    \draw[blue, very thick, -{Latex[length=2.75mm]}] (-2,2) -- (-1,2);
    \draw[blue, very thick, -{Latex[length=2.75mm]}] (-1,2) -- (-1,3.5);
    \draw[blue, very thick, -{Latex[length=2.75mm]}] (-1,3.5) -- (0,3.5);
    \draw[blue, very thick, -{Latex[length=2.75mm]}] (0,3.5) -- (0,4.25);

    \draw[red, very thick, -{Latex[length=2.75mm]}] (-1,0) -- (-1,1.5);
    \draw[red, very thick, -{Latex[length=2.75mm]}] (-1,1.5) -- (0,1.5);
    \draw[red, very thick, -{Latex[length=2.75mm]}] (0,1.5) -- (0,2);
    \draw[red, very thick, -{Latex[length=2.75mm]}] (-3,1) -- (-3,2.5);
    \draw[red, very thick, -{Latex[length=2.75mm]}] (-3,2.5) -- (-2,2.5);
    \draw[red, very thick, -{Latex[length=2.75mm]}] (-2,2.5) -- (-2,4.25);

    \draw[orange, very thick, -{Latex[length=2.75mm]}] (0,2.5) -- (0,3);
    \draw[orange, very thick, -{Latex[length=2.75mm]}] (-4,1.5) -- (-4,3);
    \draw[orange, very thick, -{Latex[length=2.75mm]}] (-4,3) -- (-3,3);
    \draw[orange, very thick, -{Latex[length=2.75mm]}] (-3,3) -- (-3,4.25);
    
    \draw[cyan, very thick, -{Latex[length=2.75mm]}] (0,0) -- (-1,0);
    \draw[cyan, very thick, -{Latex[length=2.75mm]}] (-1,0.5) -- (0,0.5);
    \draw[cyan, very thick, -{Latex[length=2.75mm]}] (-2,1) -- (-1,1);
    \draw[cyan, very thick, -{Latex[length=2.75mm]}] (-1,1.5) -- (-1,2);
    \draw[cyan, very thick, -{Latex[length=2.75mm]}] (0,2) -- (0,2.5);
    \draw[cyan, very thick, -{Latex[length=2.75mm]}] (-4,2) -- (-3,2);
    \draw[cyan, very thick, -{Latex[length=2.75mm]}] (-1,3.5) -- (-1,4.25);

    \foreach \y in {-0.5,0,0.5,1,1.5,2,2.5,3,3.5}
        \fill (0,\y) circle (2pt);

    \foreach \y in {0,0.5,1,1.5,2,2.5,3,3.5}
        \fill (-1,\y) circle (2pt);

    \foreach \y in {0.5, 1, 1.5,2,2.5,3,3.5}
        \fill (-2,\y) circle (2pt);

    \foreach \y in {1, 1.5, 2, 2.5, 3,3.5}
        \fill (-3,\y) circle (2pt);

    \foreach \y in {1.5, 2, 2.5, 3,3.5}
        \fill (-4,\y) circle (2pt);

    \draw (-0.5,4.25);
    \fill (-0.5,4.25) circle (1pt);
    \draw (-0.5,4.5);
    \fill (-0.5,4.5) circle (1pt);
    \draw (-0.5,4.75);
    \fill (-0.5,4.75) circle (1pt);    

    \draw (-2.5,4.25);
    \fill (-2.5,4.25) circle (1pt);
    \draw (-2.5,4.5);
    \fill (-2.5,4.5) circle (1pt);
    \draw (-2.5,4.75);
    \fill (-2.5,4.75) circle (1pt);    

    \draw (-5,2.75);
    \fill (-5,2.75) circle (1pt);
    \draw (-5.25,2.75);
    \fill (-5.25,2.75) circle (1pt);
    \draw (-5.5,2.75);
    \fill (-5.5,2.75) circle (1pt);    

    \draw (-5,4);
    \fill (-5,4) circle (1pt);
    \draw (-5.25,4);
    \fill (-5.25,4) circle (1pt);
    \draw (-5.5,4);
    \fill (-5.5,4) circle (1pt);    

\end{tikzpicture}
\end{subfigure}
\caption{A subdivision of the bidirected quarter-grid in the ascending (on the left) and the descending (on the right) cyclically directed quarter-grid, where the edges coloured in blue, red and orange correspond to subdivisions of the rays $R_1$, $R_2$, and $R_3$ from the bidirected quarter-grid.
    The cyan edges highlight the edges of bidirected quarter-grid between the rays.}
    \label{fig:D_1_in_D_2}
\end{figure}
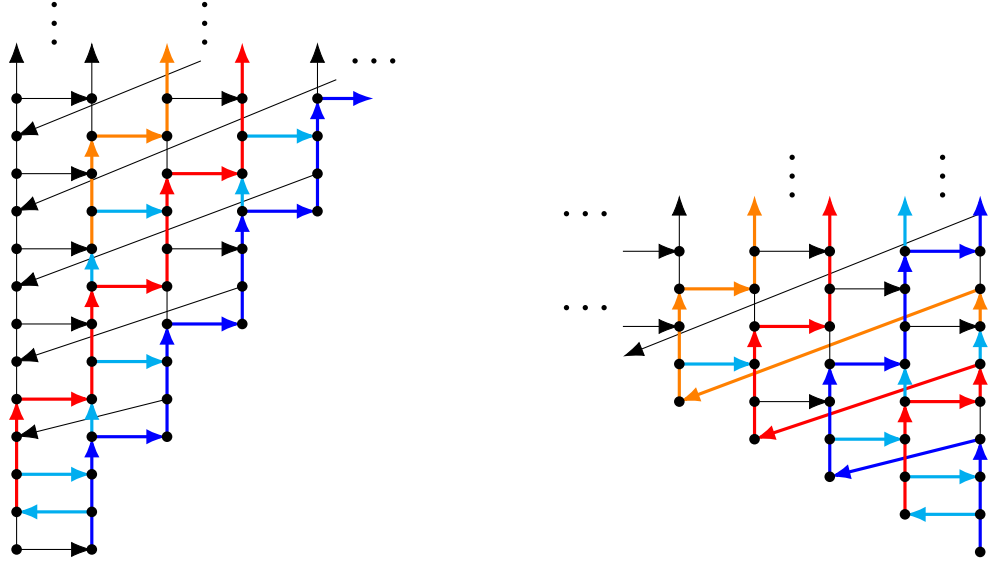

Obviously, any complete ray digraph contains a subdivision of the bidirected quarter-grid.
A way to find a subdivision of the bidirected quarter-grid in the cyclically directed quarter-grids, is indicated in Figure~\ref{fig:D_1_in_D_2}.
This finishes the proof of the theorem.
\end{proof}

Zuther \cite{Z1998}*{Theorem 3.1} proved that every digraph with an infinite increasing sequence of ends, each of which contains a ray, has a thick end as a supremum of this sequence and the digraph from Figure~\ref{fig:Zuther} as subdivision with the red ray lying in the thick end.
Using the same method how we found a bidirected quarter-grid in the cyclically directed quarter-grids, we can also find a bidirected quarter-grid in Zuther's digraph.

\begin{figure}[ht]
    \centering
\begin{tikzpicture}

    \draw[-{Latex[length=2.75mm]}] (0,0) -- (0,3.75);
    \draw[-{Latex[length=2.75mm]}] (1,0) -- (1,3.75);
    \draw[-{Latex[length=2.75mm]}] (2,0.5) -- (2,3.75);
    \draw[-] (3,1) -- (3,3.5);
    \draw[-{Latex[length=2.75mm]}] (4,1.5) -- (4,3.75);

    \foreach \y in {1,2,3}
    \draw[-{Latex[length=2.75mm]}] (0,\y) -- (1,\y);

    \foreach \y in {0.5,2.5}
    \draw[-{Latex[length=2.75mm]}] (1,\y) -- (2,{\y});

    \foreach \y in {1,2}
    \draw[-{Latex[length=2.75mm]}] (2,\y) -- (3,{\y});

    \foreach \y in {1.5,2.5}
    \draw[-{Latex[length=2.75mm]}] (3,\y) -- (4,{\y});

    \draw[-{Latex[length=2.75mm]}] (4,2) -- (4.75,2);
    \draw[-{Latex[length=2.75mm]}] (4,3) -- (4.75,3);

    \draw[red, very thick, -{Latex[length=2.75mm]}] (0,0) -- (1,0);
    \draw[red, very thick, -{Latex[length=2.75mm]}] (1,0) -- (1,1.5);
    
    \draw[red, very thick, -{Latex[length=2.75mm]}] (1,1.5) -- (2,1.5);
    \draw[red, very thick, -{Latex[length=2.75mm]}] (2,1.5) -- (2,3);

    \draw[red, very thick, -{Latex[length=2.75mm]}] (2,3) -- (3,3);
    \draw[red, very thick, -{Latex[length=2.75mm]}] (3,3) -- (3,3.75);

    \foreach \y in {0,0.5,1,1.5,2,2.5,3}
        \fill (0,\y) circle (2pt);

    \foreach \y in {0,0.5,1,1.5,2,,2.5,3}
        \fill (1,\y) circle (2pt);
    
    \foreach \y in {0.5,1,,1.5,2,2.5,3}
        \fill (2,\y) circle (2pt);

    \foreach \y in {1,1.5,2,2.5,3}
        \fill (3,\y) circle (2pt);

    \foreach \y in {1.5,2,2.5,3}
        \fill (4,\y) circle (2pt);

    \draw (0.5,3.75);
    \fill (0.5,3.75) circle (1pt);
    \draw (0.5,4);
    \fill (0.5,4) circle (1pt);
    \draw (0.5,4.25);
    \fill (0.5,4.25) circle (1pt);    

    \draw (2.5,3.75);
    \fill (2.5,3.75) circle (1pt);
    \draw (2.5,4);
    \fill (2.5,4) circle (1pt);
    \draw (2.5,4.25);
    \fill (2.5,4.25) circle (1pt);    

    \draw (4.75,3.75);
    \fill (4.75,3.75) circle (1pt);
    \draw (5,3.75);
    \fill (5,3.75) circle (1pt);
    \draw (5.25,3.75);
    \fill (5.25,3.75) circle (1pt);    

\end{tikzpicture}
    \caption{The digraph from Zuther \cite{Z1998} with the red ray being from the thick end.}
    \label{fig:Zuther}
\end{figure}
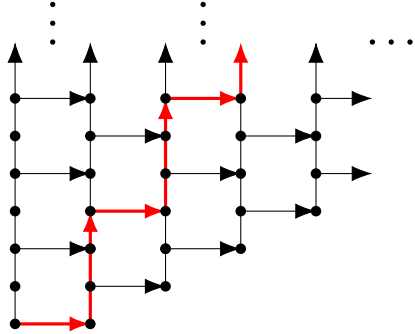

\section{Grids in thin ends}
\label{sec:thinends}

In this section we prove a result for ends of finite in-degree.
This can be done for the out-degree completely analogously, which is why we omit the details for that here.

We begin this section with definitions of bounded versions of the bidirected quarter-grid and the cyclically directed quarter-grid.
For the following definitions let $R_1, \ldots, R_n$ be $n \in \mathbb N$ pairwise disjoint rays, where $V(R_i) = r^i_1, r^i_2, \ldots$ for every $i \in \mathbb N$.

We define the \emph{hexagonal grid of width} $n \in \mathbb N$ as $\bigcup^n_{i=1}R_i$ together with the edges $r^i_jr^{i+1}_j$ for all pairs $(i, j) \in \mathbb{N} \times \mathbb{N}$ where $i = j = 1 \; (\textnormal{mod } 2)$ and $j = 1 \; (\textnormal{mod } 4)$, or $i = j = 0 \; (\textnormal{mod } 2)$ and $j = 2 \; (\textnormal{mod } 4)$, and together with the edges $r^{i+1}_jr^{i}_j$ for all pairs $(i, j) \in \mathbb{N} \times \mathbb{N}$ where $i = j = 1 \; (\textnormal{mod } 2)$ and $j = 3 \; (\textnormal{mod } 4)$, or $i = j = 0 \; (\textnormal{mod } 2)$ and $j = 0 \; (\textnormal{mod } 4)$.
See Figure~\ref{fig:H^4}~\subref{subfig:H^4} for an example of a hexagonal grid of width $4$.

Let us define the \emph{circular grid of width} $n \in \mathbb N$ as $\bigcup^n_{i=1}R_i$ together with the edges $r^i_{j+1}r^{i+1}_j$ for all odd $j \in \mathbb N$ and all $i \in \{2,\ldots, n-1\}$, the edges $r^1_jr^2_j$ for all odd $j$ and the edges $r^n_jr^1_j$ for all even $j$.
See Figure~\ref{fig:circ_3}~\subref{subfig:circ_3} for an example of a circular grid of width $3$.

\begin{figure}[ht]
    \centering
       \begin{subfigure}[b]{0.38\linewidth}
    \centering
       \begin{tikzpicture}
    \draw[-{Latex[length=2.75mm]}] (0,0) -- (0,2.75);
    \foreach \y in {0,0.5,1,1.5,2}
        \fill (0,\y) circle (2pt);

    \draw (0,-0.75) node[anchor=south] {$R_1$};

    \draw[-{Latex[length=2.75mm]}] (1,0) -- (1,2.75);
    \foreach \y in {0,0.5,1,1.5,2}
        \fill (1,\y) circle (2pt);
    
    \draw (1,-0.75) node[anchor=south] {$R_2$};

    \draw[-{Latex[length=2.75mm]}] (2,0) -- (2,2.75);
    \foreach \y in {0,0.5,1,1.5,2}
        \fill (2,\y) circle (2pt);

    \draw (2,-0.75) node[anchor=south] {$R_3$};

    \draw[-{Latex[length=2.75mm]}] (3,0) -- (3,2.75);
    \foreach \y in {0,0.5,1,1.5,2}
        \fill (3,\y) circle (2pt);
    
    \draw (3,-0.75) node[anchor=south] {$R_4$};

    \foreach \y in {0,2}
    \draw[-{Latex[length=2.75mm]}] (0,\y) -- (1,\y);
    
    \draw[-{Latex[length=2.75mm]}] (1,1) -- (0,1);
    
    \draw[-{Latex[length=2.75mm]}] (1,0.5) -- (2,0.5);
    \draw[-{Latex[length=2.75mm]}] (2,1.5) -- (1,1.5);
    
    \foreach \y in {0,2}
    \draw[-{Latex[length=2.75mm]}] (2,\y) -- (3,{\y});

    \draw[-{Latex[length=2.75mm]}] (3,1) -- (2,1);

    \draw (0.5,2.75);
    \fill (0.5,2.75) circle (1pt);
    \draw (0.5,3);
    \fill (0.5,3) circle (1pt);
    \draw (0.5,3.25);
    \fill (0.5,3.25) circle (1pt);    

    \draw (2.5,2.75);
    \fill (2.5,2.75) circle (1pt);
    \draw (2.5,3);
    \fill (2.5,3) circle (1pt);
    \draw (2.5,3.25);
    \fill (2.5,3.25) circle (1pt);    

\end{tikzpicture}
    \caption{The hexagonal grid of width $4$.}
\label{subfig:H^4}
        \end{subfigure}
       \hspace{3em}
        \begin{subfigure}[b]{0.38\linewidth}
    \centering
\begin{tikzpicture}
    \draw[-{Latex[length=2.75mm]}] (0,0.5) -- (0,2.75);
    \foreach \y in {0.5,1,1.5,2}
        \fill (0,\y) circle (2pt);

    \draw (0,2.5) node[anchor=west] {$R_1$};

    \pgfmathsetmacro{\x}{cos(-30)}
    \pgfmathsetmacro{\y}{sin(-30)}
   
    \draw[-{Latex[length=2.75mm]}] (0.5*\x,0.5*\y) -- ({2.75*\x},{2.75*\y});
    \foreach \z in {0.5,1,1.5,2}
        \fill ({\z*\x},{\z*\y}) circle (2pt);
    
    \draw ({2.5*\x + 0.5},{2.5*\y +0.5}) node[anchor=east] {$R_2$};
    
    \pgfmathsetmacro{\a}{cos(210)}
    \pgfmathsetmacro{\b}{sin(210)}
    
    \draw[-{Latex[length=2.75mm]}] (0.5*\a,0.5*\b) -- ({2.75*\a},{2.75*\b});
    \foreach \z in {0.5,1,1.5,2}
        \fill ({\z*\a},{\z*\b}) circle (2pt);

    \draw ({2.5*\a - 0.5},{2.5*\b +0.5}) node[anchor=west] {$R_3$};

    \draw[-{Latex[length=2.75mm]}] (0,0.5) -- ({0.5*\x},{0.5*\y});
    \draw[-{Latex[length=2.75mm]}] (0,1.5) -- ({1.5*\x},{1.5*\y});
    
    \draw[-{Latex[length=2.75mm]}] ({1*\x},{1*\y}) -- ({0.5*\a},{0.5*\b});
    \draw[-{Latex[length=2.75mm]}] ({2*\x},{2*\y}) -- ({1.5*\a},{1.5*\b});
    
    \draw[-{Latex[length=2.75mm]}] ({1*\a},{1*\b}) -- (0,1);
    \draw[-{Latex[length=2.75mm]}] ({2*\a},{2*\b}) -- (0,2);
    
    \draw ({-1.25*\a},{-1.25*\b});
    \fill ({-1.25*\a},{-1.25*\b}) circle (1pt);
    \draw ({-1.5*\a},{-1.5*\b});
    \fill ({-1.5*\a},{-1.5*\b}) circle (1pt);
    \draw ({-1.75*\a},{-1.75*\b});
    \fill ({-1.75*\a},{-1.75*\b}) circle (1pt);

    \draw ({-1.25*\x},{-1.25*\y});
    \fill ({-1.25*\x},{-1.25*\y}) circle (1pt);
    \draw ({-1.5*\x},{-1.5*\y});
    \fill ({-1.5*\x},{-1.5*\y}) circle (1pt);
    \draw ({-1.75*\x},{-1.75*\y});
    \fill ({-1.75*\x},{-1.75*\y}) circle (1pt);

    \draw (0,-1.25);
    \fill (0,-1.25) circle (1pt);
    \draw (0,-1.5);
    \fill (0,-1.5) circle (1pt);
    \draw (0,-1.75);
    \fill (0,-1.75) circle (1pt);    

\end{tikzpicture}
    \caption{The circular grid of width $3$.}
 \label{subfig:circ_3}

        \end{subfigure}
\caption{Hexagonal and circular grids of finite width.}
\label{fig:circ_3}
\label{fig:H^4}
\end{figure}
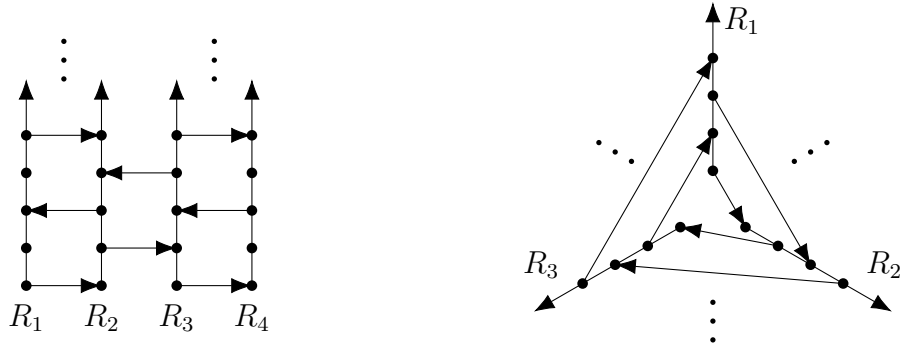

While for thick ends with only countably many rays, we found a subdivision of the bidirected quarter-grid with as many disjoint rays as the end contains, this is not expectable for thin ends with respect to a hexagonal grid of corresponding width.
However, our next result gives at least a bound on the width of a hexagonal grid or a circular grid inside a thin end, depending only on the maximum number of disjoint rays in that end.
Contrary to the situation for thick ends, we can prescribe a set of rays in that end such that the rays $R_i$ of our (hexagonal or circular) grid are from this set.

\begin{thm}
\label{thm:thinEnds}
For all $n\in\mathbb N$ there exists $k(n)\in\mathbb N$ such that in all digraphs with an end $\omega$ of in-degree at least~$k(n)$ and a set $\mathcal R$ of $k(n)$ pairwise disjoint rays in~$\omega$ there is a subdivision of the hexagonal grid of width $n$ or of the circular grid of width $n$ all whose rays $R_i$ are from~$\mathcal R$.
\end{thm}

\begin{proof}
Let $k(n)$ be the constant from Theorem~\ref{thm:Kasper} guaranteeing either the existence of a dicycle of length at least $n$ or a semi-chain with $2n$ dicycles or a $((2n-1)n+3,1)$-system of dipaths.
Let $D$ be a digraph and $\omega$ be an end of $D$ with in-degree at least $k(n)$, and let $R_1,\ldots, R_{k(n)}$ be ${k(n)}$ pairwise disjoint rays in~$\omega$.
Let $H$ be the auxiliary $\omega$-ray digraph with $R_1,\ldots, R_{k(n)}$ as vertices.
Since this digraph is strong, there exists either a dicycle of length at least $n$ in~$H$, a semi-chain of $2n$ dicycles, or a $((2n-1)n+3,1)$-system of $x$--$y$ dipaths for some vertices $x,y\in V(H)$ by Theorem~\ref{thm:Kasper}.

In the situation of a dicycle of length at least~$n$, we can cyclically connect segments of the rays on that dicycle in the same way as we did it in the proof of Theorem~\ref{thm:planarGrid} in the situation of dicycles of increasing lengths, and do this infinitely many times.
Thereby, we obtain a subdivision of the circular grid of width $n$.

In the situation of an $((2n-1)n+3,1)$-system of $x$--$y$ dipaths for $x,y\in V(H)$, we can pick one inner vertex from each but one of the $((2n-1)n+3,1)$ distinct $x$--$y$ dipaths and then completely connect segments of those rays, and repeat this infinitely many times.
This results in a digraph that contains the hexagonal grid of width $(2n-1)n+2$ and the circular grid of width $(2n-1)n+2$ as subdivisions.

In the situation of a semi-chain of $2n$ dicycles, we pick vertices from every second dicycle.
Then, we linearly connect segments of those rays and repeat this infinitely many times.
The resulting digraph contains a subdivision of the hexagonal grid of width~$n$.

By construction, the rays corresponding to the $R_i$ in the constructed subdivisions always lie in~$\mathcal R$.
\end{proof}

As corollary, we obtain that the digraph always contains a subdivision of a hexagonal grid of bounded width, see Corollary~\ref{cor:thinEnds}.
We note that this transfers a result by Stein~\cite{S2011}*{Theorem 3.2.2} from graphs to digraphs.

\begin{corollary}\label{cor:thinEnds}
For all $n\in\mathbb N$ there exists $k(n)\in\mathbb N$ such that in all digraphs with an end of in-degree at least~$k(n)$ there is a subdivision of the hexagonal grid of width $n$ all of whose rays are in that end.
\end{corollary}

\begin{proof}
Let $k(n)$ be the value from Theorem~\ref{thm:thinEnds} for finding a hexagonal grid of width $n+1$ or a circular grid of width $n+1$.
So we find one of those as a subdivision in our digraph.
By a similar argument that we find a subdivision of the bidirected quarter-grid within the cyclically directed quarter-grid, we find a subdivision of a hexagonal grid of width $n$ within a circular grid of width $n+1$, which implies the assertion.
\end{proof}

Stein's bound in \cite{S2011}*{Theorem 3.2.2} is sharp as she showed with an example \cite{S2011}*{Example 3.2.3}.
Let us modify her example so that it leads to a digraph with a unique thin end of in-degree $\frac{3}{2}n-1$ without a subdivision of the hexagonal grid of width~$n$.

\begin{example}\label{ex:thinEnds}
Let $D_{\ell}$ be a digraph on $3\ell+1$ rays: $R_0=x_0x_1\ldots$ and $R_i^k=x_0^{i,k}x_1^{i,k}\ldots$ for $1\leq i\leq 3$ and $1\leq k\leq\ell$.
For all $1\leq i\leq 3$ and $1\leq k< \ell$, we add edges $x_jx_j^{i,1}$ and $x_j^{i,k}x_j^{i,k+1}$ for all even $j\in\mathbb N$ and edges $x_j^{i,1}x_j$ and $x_j^{i,k+1}x_j^{i,k}$ for all odd $j\in\mathbb N$.
Then the $3\ell+1$ rays lie in the same end and it is easy to see that there are no more than $3\ell+1$ pairwise disjoint rays in that end. So it has in-degree $3\ell+1$.
The underlying undirected graph of $D_\ell$ is precisely the graph $Y(\ell)$ from Stein's example \cite{S2011}*{Example 3.2.3}, and thus it does not contain a subdivision of the underlying undirected graph of a hexagonal grid of width $2\ell+2$.
Thus, $D_\ell$ cannot contain a hexagonal grid of width $2\ell+2$ as a subdivision.
\end{example}

The bound from Theorem~\ref{thm:Kasper} that leads to the value of $k(n)$ in Corollary~\ref{cor:thinEnds} is much larger than $\frac{3}{2}n-1$, the value that we obtain from Example~\ref{ex:thinEnds}, which is the best lower bound that we have so far.
Thus, the optimal value for $k(n)$ remains unclear.
This motivates the following problem.

\begin{problem}
Let $n\in\mathbb N$.
Determine the smallest value $k(n)$ such that every digraph with an end of in-degree $k(n)$ contains a subdivision of the hexagonal grid of width $n$.
\end{problem}

\section{Weak immersions of bidirected quarter-grids}
\label{sec:edge grid}

In this section, we will prove a grid-theorem for edge-disjoint rays in ends of digraphs.
By \cite{HH2024+}*{Theorem 6.1}, we know that an end containing $n$ pairwise edge-disjoint rays for all $n\in\mathbb N$ contains infinitely many pairwise edge-disjoint rays.
Thus, it is natural to also ask which grid-like structures we obtain for edge-disjoint rays.

For the result of this section, we need the definition of a weak immersion.
Let $D$ and~$H$ be digraphs.
A \emph{weak immersion} of~$H$ in~$D$ is a map $\varphi$ with domain $V(H) \cup E(H)$ such that $\varphi$ restricted to $V(H)$ is injective, has its image in $V(D)$ and such that every edge $uv \in E(H)$ is mapped to a $\varphi(u)$--$\varphi(v)$ dipath in~$D$ where every two such images $\varphi(e)$ and $\varphi(f)$ for distinct $e, f\in E(H)$ are edge-disjoint.

The proof of a corresponding version of Theorem~\ref{thm:planarGrid} involving weak immersions can be obtained in this setting by following its original proof almost verbatim.
In particular, we only slightly change the definition of the auxiliary ray digraphs by allowing edge-disjoint rays to correspond to their vertices, while their edges are still determined by infinite dipath systems consisting of pairwise vertex-disjoint dipaths.
This enables us to apply the results from Section~\ref{sec:fin_digraphs}.
Thus, we obtain the following version of Theorem~\ref{thm:planarGridIntro}.

\begin{thm}
If $D$ is a digraph that contains an end~$\omega$ with infinitely many pairwise edge-disjoint (anti-)rays, then there exists a weak immersion of the (reversed) bidirected quarter-grid in~$D$ with all its (anti-)rays in~$\omega$.\qed
\end{thm}

Similarly, Theorem~\ref{thm:thinEndsImmersion} can be obtained almost verbatim from the proof of Theorem~\ref{thm:thinEnds}.

\begin{thm}\label{thm:thinEndsImmersion}
 For all $n\in\mathbb N$ there exists $k(n)\in\mathbb N$ such that in all digraphs with an end $\omega$ and a set $\mathcal R$ of at least~$k(n)$ pairwise edge-disjoint rays in~$\omega$ there is a weak immersion of the hexagonal grid of width $n$ or of the circular grid of width $n$ all whose rays $R_i$ are from~$\mathcal R$.\qed
\end{thm}

\section*{Acknowledgement}

We thank Kasper Johansen for allowing us to use his proof idea for Theorem~\ref{thm:Kasper}, which grew during discussions with Carsten Thomassen and the second author.

\bibliographystyle{amsplain}
\bibliography{grid}

\end{document}